\documentclass[titlepage,12pt]{article}
\usepackage{graphicx}%
\usepackage{amsmath,amssymb,amsfonts}%
\usepackage{amsthm}%
\usepackage{mathrsfs}%
\usepackage{xcolor}%
\usepackage{tikz}%
\usepackage{url}
\usepackage{enumitem}

\newtheorem{theorem}{Theorem}%
\newtheorem{proposition}{Proposition}%
\newtheorem{problem}{Problem}%
\newtheorem{corollary}{Corollary}%
\newtheorem{assumption}{Assumption}%
\newtheorem{example}{Example}%
\newtheorem{remark}{Remark}%
\newtheorem{definition}{Definition}%

\newcommand{\dl}{\ensuremath{d \kern -.15em l}}
\newcommand{\hatdl}{ d \hat{ \kern -.15em  l}}
\newcommand{\mesh}{\operatorname{mesh}}
\newcommand{\LimInn}{\operatorname{LimInn}}
\newcommand{\LimOut}{\operatorname{LimOut}}

\newcommand{\eto}{\,{\lower 1pt\hbox{$\rightarrow$}}\kern -11pt
	\hbox{\raise 4pt \hbox{$\, \scriptstyle e$}}\hskip7pt}
\newcommand{\hto}{\,{\lower 1pt\hbox{$\rightarrow$}}\kern -11pt
	\hbox{\raise 4pt \hbox{$\, \scriptstyle h$}}\hskip7pt}
\def\Nto{\,{\raise 1pt\hbox{$\rightarrow$}}\kern -13pt     \hbox{\lower 3pt \hbox{$\, \scriptstyle N$}}\hskip7pt}

\newcommand{\uscfcns}{\mathop{\textrm{usc-fcns}}}
\newcommand{\Lipfcns}{\mathop{\textrm{Lip-fcns}}}
\newcommand{\cdfcns}{\mathop{\textrm{cd-fcns}}}
\newcommand{\episplns}{\operatorname{e-spl}}
\newcommand{\reals}{{I\kern-.35em R}}
\newcommand{\Reals}{\overline{I\kern-.35em R}}
\newcommand{\natnums}{{{\rm l} \kern -.13em {\rm N} }}
\newcommand{\nats}{{I\kern -.35em N}}
\newcommand{\snats}{{I\kern -.29em N}}
\newcommand{\rats}{{Q\kern -.64em \raise 1pt \hbox{$\scriptstyle |$}\;\,}}
\newcommand{\srats}{{Q\kern -.56em \raise 1.2pt \hbox{$\scriptscriptstyle /$}\,}}
\newcommand{\ints}{Z\kern -.46em Z}
\newcommand{\ball}{{I\kern -.35em B}}
\newcommand{\Pro}{{I\kern-.35em P}}
\newcommand{\Ex}{{I\kern-.35em E}}
\newcommand{\argmin}{{\rm argmin}}
\newcommand{\red}[1]{{\color{black}#1}}
\newcommand{\rsF}{{\mathscr F}}
\newcommand{\hypog}{{\rm hypo}\,}
\newcommand{\epig}{{\rm epi}\,}
\def\state #1. { \noindent{\bf#1.\enspace}}

\begin{document}

\baselineskip=15pt

\begin{titlepage}
	\vglue 0.5cm
	\begin{center}
		\begin{large}
			{\bf A variational approach to a cumulative distribution function estimation problem under stochastic ambiguity}
			\smallskip
		\end{large}
		\vglue 3.0truecm
		\begin{tabular}{ccc}
			{\sl Julio Deride}  & {\sl Johannes O. Royset } & {\sl Fernanda Urrea}\\
			Faculty of Engineering & Daniel J. Epstein Department of&Laboratoire de \\
			and Science& Industrial and Systems Engineering&Math\'emathiques\\
			Universidad Adolfo Ibáñez & University of Southern California & INSA Rouen Normandie\\
			julio.deride@uai.cl & royset@usc.edu & fernanda.urrea@insa-rouen.fr 
		\end{tabular}
	\end{center}
	\vskip 1.5truecm
	\noindent 
		{\bf Abstract}.\quad We propose a method for finding a cumulative distribution function (cdf) that minimizes the distance to a given cdf, while belonging to an ambiguity set constructed relative to another cdf and, possibly, incorporating soft information.  Our method embeds the family of cdfs onto the space of upper semicontinuous functions endowed with  the hypo-distance.  In this setting, we present an approximation scheme based on epi-splines, defined as piecewise polynomial functions, and use bounds for estimating the hypo-distance.  Under appropriate hypotheses, we guarantee that the cluster points corresponding to the sequence of minimizers of the resulting approximating problems are solutions to a limiting problem.  We describe a  large class of functions that satisfy these hypotheses.  The approximating method produces a linear-programming-based  approximation scheme, enabling us to develop an algorithm from off-the-shelf solvers. The convergence of our proposed approximation is illustrated by numerical examples for the bivariate case.
	\vskip 1.5truecm
	\halign{&\vtop{\parindent=0pt
			\hangindent2.5em\strut#\strut}\cr
		{\bf Keywords}: \ Variational analysis, hypo-distance, epi-convergence, upper semicontinuous functions, \hglue 1.50cm cumulative distribution functions, epi-splines.\hfill\break\cr
		{\bf Date}:\quad \ \today \cr}
	\end{titlepage}
	\baselineskip=15pt

\section{Introduction}
\label{intro}

We are interested in studying the estimation problem of finding a cumulative distribution function (cdf) that minimizes the distance to a given cdf, while also maintaining proximity to a second cdf, possibly incorporating soft information with respect to the shape of the desired function.   Our primary motivation comes from estimation problems under stochastic ambiguity \cite{royset2017variational} and, particularly, distributional robust optimization (\cite{wiesemann2014distributionally} for a general convex formulation).

Our method is based on posing the problem as a mathematical program on a functional space, representing the cdf functions, and  proposing a metric in order to measure distances between the functions. Thus, this estimation problem corresponds to an infinite-dimensional optimization problem,  and  we develop a convergent approximation scheme, based on solving simpler problems sequentially.

In order to describe our problem, let us consider the Euclidean space $\left(\reals^{m},\|\cdot\|_{\infty}\right)$ and let $S\subset\reals^{m}$ be a closed subset \red{containing 0 (for convenience); in Section~\ref{sec:thms} we will also assume that $S$ is compact.  This choice of norm simplifies computations involving hypo-distance, a concept we will introduce shortly for comparing functions.} It is possible to pose the problem in all its generality in a metric functional space containing the set of cdfs defined over $S$, denoted by $(\rsF, \dl)$.  We are interested in studying the following constrained estimation problem under stochastic ambiguity: given two cdfs 
$F_0$ and $G_0$ (in $\rsF$), we aim to find a function that minimizes the distance to $F_0$, must lie in an ambiguity set described by a ball centered at $G_0$, and may also satisfy other requirements. 

Formally, this leads to the {\em constrained estimation problem under stochastic ambiguity}: 
\[\mbox{find } \hat F \in \argmin_{F \in \rsF} \left\{ \dl(F,F_0) \ \mid \ \dl(F,G_0) \leq \delta \right\}.\]
\red{Here, $\delta > 0$ is a parameter that controls the size of the ambiguity set}.

Note that this problem relates to the family of constrained $M$-estimation problems as the ones described in \cite{royset2020variational}.  The functional space $\rsF$ embeds certain topological properties of cdfs by being a subset of the upper semicontinuous (usc) functions. We let $\dl$ be the hypo-distance.  Note that we can directly incorporate shape (soft) information into this problem formulation.  This includes information such as continuity, stochastic dominance, quantile-type, moment information, pointwise bounds, and, in general, any constraint that defines a closed set with respect to $(\rsF, \dl)$.  For more details, see \cite{royset2017variational,royset2019lopsided}.

Estimation of distribution and density functions can be approached using various techniques.  For example, kernel density estimation  \cite{WandJones94kernel,Hill85kernelcdf} poses the problem of fitting a kernel function to a particular data set, and deducing the corresponding cdf based on estimation of the probability density function. Alternatively,  one can compute the empirical distribution function directly from data as an estimate of a cdf; see for example \cite{ShorackWellner09empirical}.  Another approach comes from smoothing techniques, where a cdf is approximated by smooth functions, such as cubic splines, local regression, and penalized regression \cite{silverman1986density}.  Non-parametric maximum likelihood estimation is another estimation technique  \cite{van1996weak}, by maximizing the likelihood function under appropriate hypotheses.  Related to the cdf estimation problem, one can consider quantile regression  \cite{koenker2001quantile}, where the resulting cdf follows from estimation of its quantiles.  Methods based on Bayesian estimation have also been used in the literature.  The monograph \cite{albert2009bayesian} provides a comprehensive introduction to these techniques, including the application of  Markov chain Monte Carlo (MCMC) methods, such as Gibbs sampling and Metropolis-Hasting.  These approaches aim to find an estimate ``near" $F_0$, typically an empirical cdf, but fail to account for the present of a secondary concern: proximity to $G_0$.

Although alternative metrics like those discussed in \cite{dudley2010distances}, \cite{kuzmenko2019kantorovich}, and the $L^1$-metric in \cite{haskell2022cdf} have been explored for measuring the distance between cdfs, our focus remains on the hypo-distance, as a possibility less explored.  We propose a scheme of approximations for computing the hypo-distance. This scheme is based on bounds introduced in \cite{attouch1991quantitative} (see \cite{rockafellar2009variational} and \cite{royset20stability} for further references).  \red{These bounds rely on the application of a truncated hypo-distance, depending on a truncation radius $\rho$ that provides theoretical advantages and numerical tractability, while approximating the true underlying problem.}  By utilizing these approximations, we propose a finite-dimensional approximation technique that employs epi-splines \cite{royset2016multivariate,royset2016erratum}. In this technique, usc \red{functions} are approximated by polynomials over triangular refinements of the domain.  These approximations pave the way for a solution algorithm that involves solving linear programs and provide the flexibility to incorporate soft information into the solution. 

The article is organized as follows: Section~\ref{sec:notandbg} describes the notation used in this manuscript and provides relevant results related to usc functions, the hypo-distance, cdfs, and epi-splines.   In Section~\ref{sec:thms}, we formulate a mathematical program that serves as a computationally attractive surrogate for the constrained estimation problem under stochastic ambiguity. We also develop convergence results. Section~\ref{sec:Numerical}  contains algorithmic details along with two numerical examples in two dimensions that demonstrate the viability of our approximation scheme.

\section{Preliminaries}\label{sec:notandbg}
In what follows, we consider $\reals$ the set of real numbers, and let $\Reals$ be the set of extended real numbers, including the values $-\infty$ and $\infty$.  Let $S$ be a closed subset of $\reals^m$ endowed with the $\|\cdot\|_{\infty}$-norm. \red{We choose $S$ to contain $0 \in \reals^m$ for simplicity, as it serves as a convenient reference point in our analysis. This choice can be easily generalized to any other element $x\in S$ with only trivial changes to the results.}

For  a function $f: S \rightarrow \Reals$, we define
$$
\liminf _{x^{\prime} \rightarrow x} f\left(x^{\prime}\right):=\lim _{\delta \downarrow 0}\left[\inf _{x^{\prime} \in \ball(x, \delta)} f\left(x^{\prime}\right)\right], \qquad \limsup _{x^{\prime} \rightarrow x} f\left(x^{\prime}\right):=\lim _{\delta \downarrow 0}\left[\sup _{x^{\prime} \in  \ball(x, \delta)} f\left(x^{\prime}\right)\right],
$$
where, for $\delta>0$,  $\ball(x,\delta)=\{x'\in S\mid \|x'-x\|_{\infty}\leq \delta\}$. For $f:S\to\Reals$, we define the hypograph of $f$ as the set
\[\hypog f:= \left\{\left(x, x_{0}\right) \in S \times \reals\,\mid\, f(x) \geq x_{0}\right\}.\]

We say that the function $f$ is upper semicontinuous (usc) at a point $\bar{x}$ if 
\begin{equation*}
	\limsup _{x \rightarrow \bar{x}} f(x) \leq f(\bar{x}).
\end{equation*}

If $f$ is usc at every point in $S$ we say that it is usc.  Equivalently, a function is usc if its hypograph is a closed set (in $S\times \reals$).

We introduce the following sets of functions: 
\begin{eqnarray}
	\label{set_uscfcns}\uscfcns(S)&:=&\{f: S \rightarrow [0,1]\mid f\text{ is usc }\},\\ 
	\label{set_uscfcns+}{\uscfcns}_{+}(S)&:=&\left\{f: S \rightarrow[0,1] \mid f \text { is usc and nondecreasing }\right\},\\
	\label{set_lipfcns}{\Lipfcns}_{\kappa}(S)&:=&\left\{f: S \rightarrow[0,1] \mid f \text { is Lipschitz of modulus }\kappa\right\},
\end{eqnarray}
\sloppy where a function $f$ is Lipschitz of modulus $\kappa$ if there exists a finite $\kappa$ such that $\kappa\geq\sup\left\{{|f(x) - f(x')|}/{\|x-x'\|_\infty} \,\mid\, x,x'\in S, x\neq x'\right\}$.
For a function $f:S\to\Reals$, we also define its epigraph as the set 
\[\epig f:= \left\{\left(x, x_{0}\right) \in S \times \reals\mid f(x) \leq x_{0}\right\}.\]

We say that the function $f$ is lower semicontinuous (lsc) at $\bar{x}$ if $\liminf_{x \rightarrow \bar{x}} f(x)\geq f(\bar{x})$.  If the function is lsc at every point, we say that it is lsc.  Equivalently, $f$ is lsc if and only if its epigraph is a closed set.  

Additionally, as $\hypog f\subset S\times\reals$, we endowed this space with the norm
\begin{equation}
	\label{ec_norm_S}
	\left\|\left(x, x_{0}\right)\right\|_{\mathbb{S}}:=\max \left\{\|x\|_{\infty},\left|x_{0}\right|\right\} \text { for }\left(x, x_{0}\right) \in S \times \reals.
\end{equation}

Thus, a $\|\cdot\|_{\mathbb{S}}$-ball, with radius $r$ and centered at $\bar{x}=\left(x, x_{0}\right) \in S \times \reals$, is denoted by
\begin{equation}
	\label{ec_S}
	\mathbb{S}(\bar{x}, r):=\left\{\bar{y} \in S \times \reals\,|\,\|\bar{x}-\bar{y}\|_{\mathbb{S}} \leq r\right\}.
\end{equation}
For brevity, we write $\mathbb{S}$ in place of $\mathbb{S}(0,1)$.  We define the distance between a point $\bar{x} \in S \times \reals$ and a set $C \subset S \times \reals$ as
\begin{equation}
	\label{ec_dist_S}
	\operatorname{dist}(\bar{x}, C):=\inf \left\{\|\bar{x}-\bar{y}\|_{\mathbb{S}}\mid \bar{y} \in C\right\},
\end{equation}
with the convention that if $C$ is empty, the distance is defined to be infinity.  For a subset $C\subset\reals^m$ we denote its closure by $\operatorname{cl}\,C$.

Let $\nats=\{1,2,\ldots\}$ and $\nats_0=\{0\}\cup \nats$.

For a sequence of sets $\{C^\nu\subset\reals^m,\,\nu\in\nats\}$, we define the limiting sets 
\begin{align*}
	\operatorname{LimInn}C^\nu&:=\{x\in\reals^m\,\mid\,\exists x^\nu\in C^\nu\to x\}\quad \text{and}\\
	\operatorname{LimOut}C^\nu&:=\{x\in\reals^m\,\mid\,\exists N\in \mathcal{N}_\infty^\#\,\text{ and }\,x^\nu\in C^\nu{\Nto}x\},
\end{align*}
where $\mathcal{N}_\infty^\#$ is the set of all infinite collections of increasing numbers from $\nats$, and $x^\nu \Nto x$ represents the fact that $x$ is a cluster point corresponding to the subsequence specified by the index set $N$.  We say that $\{C^\nu\subset\reals^m,\,\nu\in\nats\}$ set-converges to $C\subset\reals^m$ if $\operatorname{LimInn}C^\nu=\operatorname{LimOut}C^\nu=C$ (for details, see \cite{Royset:2020aa}).

\subsection{Metric on the space of usc functions}

For $f,g\in\uscfcns(S)$, we define the hypo-distance between them as 
\begin{equation}
	\label{attouch_wets_dist}
	\dl\left(f, g\right):=\int_{0}^{\infty} \dl_{\rho}\left(f, g\right) e^{-\rho}\,d \rho,
\end{equation}
where for a parameter $\rho>0$, the truncated $\rho$ hypo-distance or simply \emph{$\rho$-distance}\footnote{Although this is just a pseudo-distance, in the sense that it satisfies the properties for being a metric, except that the distance between two different elements might be zero.} is given by
\begin{equation}
	\label{d-rho-def}
	\dl_{\rho}\left(f, g\right):=\max \left\{\left| \operatorname{dist}(\bar{x}, \hypog f)-\operatorname{dist}\left(\bar{x}, \hypog g\right) \right| \ \mid\ \|\bar{x}\|_{\mathbb{S}} \leq \rho\right\}.
\end{equation}
Note that the definition of hypo-distance can be seen as the set distance, in the Attouch-Wets sense, between the corresponding hypographs, see \cite{Royset:2020aa} for more details.

As a supplement of the $\rho$-distance, we consider the approximated $\rho$-distance or \emph{hat-distance} following the approach in \cite[Ch.\,4]{rockafellar2009variational}, defined  as
\[	\hatdl_{\rho}(f, g):=\inf\left\{\eta \geq 0\,\left|\begin{array}{l}
	\hypog f\cap \rho\mathbb{S}\subset \hypog g+\eta \mathbb{S}\\
	\hypog g\cap \rho \mathbb{S}\subset \hypog f+\eta \mathbb{S}
\end{array}\right.\right\}.\]
This hat-distance  can be stated in a more direct fashion in the context of monotone nondecreasing functions, relevant to our estimation problem.  If $f,g\in\uscfcns_{+}(S)$ we obtain the more explicit expression
\begin{align}
	\label{d-rho-hat-def}
	\hatdl_{\rho}(f, g):=\inf \left\{  \eta\geq 0 \left| \begin{array}{l} 
		\displaystyle \max_{y\in Y_\eta(x)} g(x+y)+\eta \geq \min \{f(x), \rho\} \\ 
		\displaystyle \max_{y\in Y_\eta(x)} f(x+y)+\eta \geq \min \{g(x), \rho\}
	\end{array},\right.\ \forall x \in \rho\ball\right\},
\end{align}
\red{for $Y_\eta(x):=\{y \in \reals^m\,\,\vert\,\, x+y\in S,\, \eta \mathbf{1} \geq y \geq 0\}$, where $\mathbf{1}$ denotes the all-ones vector in $\reals^m$} and $\ball:=\left\{y\in S\,\vert\,\|y\|_\infty\leq 1\right\}$.  The expression in \eqref{d-rho-hat-def} follows as a simplification of the Kenmochi condition: see for example \cite[Prop.\,6.58]{royset2021optimization}).  Thus, the computation of the hat-distance in \eqref{d-rho-hat-def} provides a computationally attractive surrogate for the hypo-distance. 

Finally, we say that a sequence $\{f^\nu\in\uscfcns(S)\}_{\nu\in\nats}$ {\it hypo-converges} to $f$, denoted by $f^\nu\hto f$, if $\dl(f^\nu,f)\to 0$ as $\nu\to\infty$.  

The following section discusses several bounds for the $\rho$-distance and hat-distance, which will serve as stepping stones towards our approximation scheme.

\subsection{Relevant bounds for the hypo-distance}
Our approximation scheme relies on substituting the hypo-distance by the hat-distance, combined with a relaxation of the hypo-distance using distance bounds.  Let us state the following propositions.

\begin{proposition}\label{prop:less1}
	If $f,g \in \uscfcns(S)$, then $\dl_\rho(f,g)\leq 1$ for every $\rho\geq 0$. Moreover, $\dl(f,g)\leq 1$.
\end{proposition}

The proof of Proposition~\ref{prop:less1} follows directly from the definitions.  The next results, Proposition~\ref{prop:est-rho-hypo-dist}~and~\ref{prop:estimates-hypodistance}, restate and refine arguments in \cite[Prop.\,3.1]{royset2018approximations} by considering the specific setting of $\uscfcns(S)$.

\begin{proposition}[Estimates of $\rho$-distance and hat-distance]
	\label{prop:est-rho-hypo-dist}
	If $f,g \in {\uscfcns}(S)$, then
	$$
	\hatdl_{\rho}\left(f,g\right) \leq \dl_{\rho}\left(f, g\right) \leq \hatdl_{2\rho}\left(f, g\right) \quad \text { for }\ \rho \geq 0.
	$$
\end{proposition}

\begin{proof}
	Let $C,D\subset S\times\reals$ be two closed sets such that $0 \in C\cap D$. Let $\varepsilon>0$,  $\rho>0$ and $\bar\rho \geq 2 \rho$. We first show that
	\begin{equation}
		\label{1st_implication}
		\operatorname{dist}(\cdot, D) \leq \operatorname{dist}(\cdot, C)+\varepsilon\text{ on }\rho \mathbb{S}\text{  implies that }C \cap \rho \mathbb{S} \subset D+\varepsilon \mathbb{S}.
	\end{equation}
	Note that for every $\bar{x} \in C \cap \rho \mathbb{S}$ such that $\operatorname{dist}(\bar{x}, D) \leq \varepsilon$, as $D$ is closed, we have that $C \cap \rho \mathbb{S} \subset D+\varepsilon \mathbb{S}$.
	Second, we establish that 
	\begin{equation}
		\label{2nd_implication}
		C \cap \bar\rho \mathbb{S} \subset D+\varepsilon \mathbb{S}\text{ implies that } \operatorname{dist}(\cdot, D) \leq \operatorname{dist}(\cdot, C)+\varepsilon\text{ on }\rho \mathbb{S}.
	\end{equation}
	For any $\bar{x} \in S\times \reals$,
	$$
	\begin{aligned}
		\operatorname{dist}\left(\bar{x}, C \cap \bar\rho \mathbb{S}\right) & \geq \operatorname{dist}(\bar{x}, D+\varepsilon \mathbb{S})\\
		&=\inf \left\{\|(\bar{y}+\varepsilon \bar{z})-\bar{x}\|_{\mathbb{S}}: \bar{y} \in D, \bar{z} \in \mathbb{S}\right\} \\
		& \geq \inf \left\{\|\bar{y}-\bar{x}\|_{\mathbb{S}}-\varepsilon\|\bar{z}\|_{\mathbb{S}}: \bar{y} \in D, \bar{z} \in \mathbb{S}\right\}\\
		&=\operatorname{dist}(\bar{x}, D)-\varepsilon.
	\end{aligned}
	$$
	Thus, $\operatorname{dist}(\cdot, D) \leq \operatorname{dist}\left(\cdot, C \cap \bar\rho \mathbb{S}\right)+\varepsilon$ on $S\times \reals$. It remains to establish that 
	\begin{equation}
		\label{ec3_prop2}
		\operatorname{dist}(\bar{x}, C \cap \left.\bar\rho\mathbb{S}\right)=\operatorname{dist}(\bar{x}, C)\text{ when }\bar{x} \in \rho \mathbb{S} \text{ and } \bar\rho \geq 2 \rho.
	\end{equation}
	Naturally we have $\operatorname{dist}(\bar{x}, C \cap \left.\bar\rho\mathbb{S}\right) \geq \operatorname{dist}(\bar{x}, C)$, for the remainder let $\bar{x} \in \rho \mathbb{S}$ and $\bar{y} \in$ $\argmin_{c \in C}\left\|\bar{x}-c\right\|_{\mathbb{S}}$, which exists since $C$ is closed. Implication \eqref{ec3_prop2} is established if $\bar{y} \in \bar\rho\mathbb{S}$. This is indeed the case because 
	$$\|\bar{y}\|_{\mathbb{S}} \leq\|\bar{x}\|_{\mathbb{S}}+\|\bar{y}-\bar{x}\|_{\mathbb{S}} $$
	with $\|\bar{y}-\bar{x}\|_{\mathbb{S}}=\operatorname{dist}(\bar{x}, C) \leq \operatorname{dist}(\bar{x}, 0) = \|\bar x\|_{\mathbb{S}}$ 
	since $0 \in C $ and consequently 
	$$\|\bar{y}\|_{\mathbb{S}} \leq 2\|\bar{x}\|_{\mathbb{S}} \leq 2 \rho\leq \bar\rho.$$ 
	
	Let $f,g\in\uscfcns(S)$.  In order to compute $\dl_\rho(f,g)$, we apply the first implication \eqref{1st_implication} with $C=\hypog f$ and $D=\hypog  g$, and then with $C=\hypog g$ and $D=\hypog f$.  Note that $(0,0)\in \hypog f$ as $0\in S$ and $f(x)\geq 0, \forall x\in S$. Analogously, $(0,0)\in \hypog g$.  For $\varepsilon = \eta > 0$, we obtain
	$$  |\operatorname{dist}(\cdot,\hypog f)- \operatorname{dist}(\cdot,\hypog g)| \leq \eta \text{ on } \rho\mathbb{S},$$
	which implies that
	$$ \hypog f \cap \rho \mathbb{S} \subset \hypog g+\eta \mathbb{S} \ \text{ and } \hypog g \cap \rho \mathbb{S} \subset \hypog f+\eta \mathbb{S},$$
	in turn $\hatdl_{\rho}\left(f, g\right) \leq \dl_{\rho}\left(f, g\right)$. By following an analogous procedure with the second implication in \eqref{2nd_implication}, we obtain that
	$\hatdl_{\rho}\left(f, g\right) \geq \dl_{\bar\rho}\left(f, g\right)$ for $\bar\rho\geq 2\rho$.
\end{proof}

\begin{proposition}[Estimates of hypo-distance and hat-distance]
	\label{prop:estimates-hypodistance}
	If $f, g \in$ $\uscfcns_{+}(S)$, then for any \red{$\rho \in[0, \infty)$},
	$$
	\hatdl_{\rho}(f, g) e^{-\rho} \leq \dl(f, g) \leq e^{-\rho}+\left(1-e^{-\rho}\right) \hatdl_{2 \rho}(f, g).
	$$
\end{proposition}

\begin{proof}
	From the definition of the hypo-distance,
	$$
	\dl\left(f, g\right)=\int_{0}^{\rho} \dl_{\tau}\left(f, g\right) e^{-\tau} d\tau +\int_{\rho}^{\infty} \dl_{\tau}\left(f, g\right) e^{-\tau} d \tau.
	$$
	Since $\dl_{\tau}\left(f, g\right)$ is nondecreasing as $\tau$ increases and $0<\tau<\rho$, we have that
	\begin{equation}
		\label{eq1esthypo}
		\dl_{0}\left(f, g\right) \int_{0}^{\rho} e^{-\tau} d\tau \leq \int_{0}^{\rho} \dl_{\tau}\left(f, g\right) e^{-\tau} d \tau \leq \dl_{\rho}\left(f,g\right) \int_{0}^{\rho} e^{-\tau} d \tau  
	\end{equation}
	and
	\begin{equation}
		\label{eq2esthypo}
		\int_{\rho}^{\infty} \dl_{\rho}\left(f, g\right)e^{-\tau} d \tau \leq \int_{\rho}^{\infty} \dl_{\tau}\left(f, g\right) e^{-\tau} d \tau.
	\end{equation}
	Because of $f,g \in \operatorname{usc-fcns_+(S)}$, $\dl_\tau(f,g) \leq 1$, then
	\begin{equation}
		\label{eq4esthypo}
		\int_{\rho}^{\infty} \dl_{\tau}\left(f, g\right) e^{-\tau} d \tau\leq \int_{\rho}^{\infty} e^{-\tau} d \tau =  e^{-\rho}.
	\end{equation}
	Carrying out the integrations in \eqref{eq1esthypo} and proceeding equally on \eqref{eq2esthypo}, we obtain that
	\begin{equation}
		\left(1-e^{-\rho}\right)\left|d^f-d^g\right|+e^{-\rho} \dl_{\rho}\left(f, g\right) \leq\dl\left(f, g\right)\leq\left(1-e^{-\rho}\right) \dl_{\rho}\left(f, g\right)+e^{-\rho},
	\end{equation}
	where $d^f = \operatorname{dist}(0,\hypog f)$ (analogous for $g$). \red{Since $0\in S$, the domains of $f$ and $g$ share at least one element.  Then, from the definition of the hypograph, both $f(0)\geq 0$ and $g(0)\geq 0$ hold true, implying $d^f =d^g=0$.} Consequently Proposition \ref{prop:est-rho-hypo-dist} gives the result.
\end{proof}

In the next definitions, we describe the partition of the domain that we use in our approximation scheme, along with the refinement procedure, and the appropriate approximations to the hat-distance in this setting.

In what follows, we restrict the set $S$ to be a rectangular region.  Henceforth rectangular $S$ means that $S=[\alpha_1,\beta_1]\times\cdots\times[\alpha_m,\beta_m]$, where $\alpha_i$ and $\beta_i$ possibly being $-\infty$ and $+\infty$ respectively.  

\begin{definition}[Box partition]
	For $\rho\geq 0$, a box partition $\mathcal{R}$ of a rectangular $S$ is a collection $\mathcal{R}=\left\{R_{1},\dots, R_{N}\right\}$ of subsets of the form $R_{k}=\left(l_{1}^{k}, u_{1}^{k}\right) \times \ldots \times\left(l_{m}^{k}, u_{m}^{k}\right)$ with $R_{k} \cap R_{k^{\prime}}=\emptyset$ for $k \neq k^{\prime}$ and $\cup_{k=1}^{N} \operatorname{cl} R_{k}=S$.  We define the mesh-size associated with $\mathcal{R}$ as $\mesh(\mathcal{R}):=\max \left\{u_{j}^{k}-l_{j}^{k} \mid k=1, \ldots, N, j=1, \ldots, m\right\}$
\end{definition}

Note that we can define triangular partitions of $S$ likewise, but we omit the details here.  Triangular partitions are useful in the numerical implementation.

\begin{definition}[Infinite refinement] 
	For $\rho>0$, a sequence $\left\{\mathcal{R}^{\nu}\right\}_{\nu=1}^{\infty}$ of partitions of a rectangular $S$, with $\mathcal{R}^{\nu}=\left\{R_{k}^{\nu}\right\}_{k=1}^{N^{\nu}}$, is an infinite refinement if	for every $x \in S$ and $\varepsilon>0$, there exists $\bar{\nu} \in \nats$  and $\delta\in(0,\varepsilon)$ such that $R_{k}^{\nu} \subset \ball(x, \varepsilon)$ for every $\nu \geq \bar{\nu}$ and $k$ satisfying $R_{k}^{\nu}\cap \ball(x,\delta)\neq\emptyset$.
\end{definition}

\begin{definition}[Functions $\eta_{\rho}^{+}$, $\eta_{\rho}^{-}$] \label{def:eta+-}
	For a partition $\mathcal{R}$ of a rectangular $S$, and functions $f,g\in\uscfcns(S)$, we define the following quantities
	\begin{eqnarray*}
		\eta_{\rho}^{+}(f,g)&:=&\inf \left\{\eta\geq 0 \ \left|\,\, \begin{array}{ll}
			\displaystyle \max_{y\in Y_\eta(l^{k})}g(l^{k}+y)+\eta &\geq \min \left\{f(u^{k}), \rho\right\}\\
			\displaystyle \max_{y\in Y_\eta(l^{k})}f(l^{k}+y)+\eta &\geq \min \left\{g(u^{k}), \rho\right\}
		\end{array}\right.,\forall  k=1, \ldots, N\right\}\\
		\eta_{\rho}^{-}(f,g)&:=&\inf \left\{\eta\geq 0 \ \left|\,\, \begin{array}{ll}
			\displaystyle \max_{y\in Y_\eta(l^{k})} g(l^{k}+y)+\eta &\geq \min \left\{f(l^{k}), \rho\right\}\\
			\displaystyle \max_{y\in Y_\eta(l^{k})} f(l^{k}+y)+\eta &\geq \min \left\{g(l^{k}), \rho\right\}
		\end{array},\right. \forall  k=1, \ldots, N\right\}.
	\end{eqnarray*}
\end{definition}

We will show that the quantities $\eta^{+}_{\rho}$ and $\eta^{-}_{\rho}$ offer computational advantages over the direct computation of the hat-distance, which generally is difficult. In our overall scheme, they will be the fundamental pieces for the solution of the estimation problem. In the next result, we establish how these quantities relate to the hat-distance.

\begin{theorem}[Approximation of hat-distance] 
	\label{thm:app-hatdistance} Let $f, g \in {\uscfcns}_{+}(S)$, $\rho \geq 0$, and consider a box partition $\mathcal{R}$ of a rectangular $S$. Then  
	$$
	\eta_{\rho}^{-}(f,g) \leq \hatdl_{\rho}(f,g) \leq \eta_{\rho}^{+}(f,g).
	$$
	Moreover, if $f, g$ are also Lipschitz continuous with modulus $\kappa$, i.e., $f,g\in{\Lipfcns} _{\kappa}(S)$, then $\eta_{\rho}^{+}(f,g)-\eta_{\rho}^{-}(f,g) \leq \kappa \mesh (\mathcal{R})$.
\end{theorem}

\begin{proof}
	The lower bound comes from the definition of $\eta_\rho^{-}(f,g)$ as a minimization problem, and the  fact that considering a finite number of constraints is a relaxation with respect to the minimization defining $\hatdl_{\rho}(f,g)$ as stated in \eqref{d-rho-hat-def}. The upper bound follows from the fact that if the constraints in the definition of $\eta_\rho^{+}(f,g)$ hold, then for any point $x \in S$ and $\eta \geq 0$ there exists an index $k$ of an element in the box partition, $y\in Y_\eta(x)$ and $y^k\in Y_\eta(l^k)$ such that
	\begin{align}
		\label{eq:1}
		f(x+y)+\eta \geq f\left(l^{k}+y^k\right)+\eta \geq \min \left\{g(u^{k}), \rho\right\} \geq \min \left\{g(x), \rho\right\} .
	\end{align}
	
	An analogous argument holds when reversing the roles of the functions $f$ and $g$. Thus, if $\eta$ satisfies the constraints in \eqref{eq:1}, it will also satisfy the constraints in the definition of $\hatdl_{\rho}(f, g)$ in \eqref{d-rho-hat-def}.
	
	Finally, if $f, g$ satisfy the Lipschitz condition with modulus $\kappa$, we consider the difference between the upper and lower bounds. Suppose that $\eta \geq 0$ is such that there exists $y^k\in Y_\eta(l^k)$ that satisfies
	\begin{align*}
		f(l^{k}+y^k)+\eta \geq \min \left\{g(l^{k}), \rho\right\} \quad \forall \ k=1, \ldots, N.
	\end{align*}
	Since $g$ is nondecreasing and Lipschitz continuous with modulus $\kappa$
	\begin{align*}
		g(l^{k}) \geq g(u^{k})-\kappa \mesh(\mathcal{R}) \quad \forall \ k=1, \ldots, N.
	\end{align*}
	
	We then also have that for $\eta^{\prime}=\eta+\kappa \mesh(\mathcal{R})$, there exists $y^\prime\in Y_{\eta^\prime}(l^k)$ such that
	\begin{align*}
		f(l^{k}+y^{\prime})+\eta^{\prime} \geq f(l^{k}+y^k)+\eta^{\prime} &\geq \min \left\{g(l^{k}), \rho\right\}+\kappa \mesh(\mathcal{R}) \\
		& \geq \min \left\{g(l^{k})+\kappa \mesh(\mathcal{R}), \rho\right\} \\
		&\geq \min \left\{g(u^{k}), \rho\right\}
	\end{align*}
	for all $k=1, \ldots, N$. A similar argument establishes the result, after reversing the roles of $f$ and $g$.
\end{proof}

The Lipschitz case of Theorem~\ref{thm:app-hatdistance} provides control of the bound between the two quantities $\eta_{\rho}^{+}(f,g)$ and $\eta_{\rho}^{-}(f,g)$, depending on the mesh size of the partition. In particular, when applied to an infinite refinement $\{\mathcal{R}^\nu\}$ with $\mesh(\mathcal{R}^\nu)\to 0$, these two quantities converge and coincide in the limit. As a result, computing these approximate $\eta$-quantities facilitates the estimation of the hat-distance, $\hatdl_{\rho}(f,g)$

\begin{remark}[Differences between $\dl, \dl_\rho$]
	From the definition of the hypo-distance, we note that each truncated $\rho$-distance is weighted by $e^{-\rho}$, thus for larger values of $\rho$, $\dl_\rho(f,g)$ have a decreasing contribution in the computation of the final hypo-distance. 
	
	For example, in the context of cdfs defined on the interval $[0,1]$ (see definition and properties in Section~\ref{ssec:cdf}), let $G$ be the cdf of $\delta_{1/2}$\footnote{Dirac's distribution at $x_0$, i.e $\delta_{x_0}=1$ if $x=x_0$ and $\delta_0=0$ elsewhere.}, and $F$ be the cdf of $\delta_1$, as is depicted in Figure~\ref{fig:remark2.13}. 
	
	\begin{figure}[ht!]
		\centering
		\tikzset{every picture/.style={line width=0.75pt}} 
		\begin{tikzpicture}[x=0.75pt,y=0.75pt,yscale=-1,xscale=1]
			\draw    (16,200) -- (462,200) ;
			\draw [shift={(465,200)}, rotate = 180] [fill={rgb, 255:red, 0; green, 0; blue, 0 }  ][line width=0.08]  [draw opacity=0] (8.93,-4.29) -- (0,0) -- (8.93,4.29) -- cycle    ;
			\draw [shift={(13,200)}, rotate = 0] [fill={rgb, 255:red, 0; green, 0; blue, 0 }  ][line width=0.08]  [draw opacity=0] (8.93,-4.29) -- (0,0) -- (8.93,4.29) -- cycle    ;
			\draw    (194,21) -- (194,202) ;
			\draw [shift={(194,18)}, rotate = 90] [fill={rgb, 255:red, 0; green, 0; blue, 0 }  ][line width=0.08]  [draw opacity=0] (8.93,-4.29) -- (0,0) -- (8.93,4.29) -- cycle    ;
			\draw  [dash pattern={on 4.5pt off 4.5pt}]  (22,198.5) -- (254.2,198) ;
			\draw  (22,198.5) -- (303.82,198.5) ;
			\draw    (306.8,80) -- (461,81) ;
			\draw  [dash pattern={on 4.5pt off 4.5pt}]  (254.4,80.4) -- (391.6,79.9) ;
			\draw  [color={rgb, 255:red, 0; green, 0; blue, 0 }  ,draw opacity=1 ][line width=0.75]  (157,226) -- (227,226) -- (227,166) -- (157,166) -- cycle ;
			(314,315) -- (74,315) -- cycle ;
			\draw  [color={rgb, 255:red, 0; green, 0; blue, 0 }  ,draw opacity=1 ][line width=0.75]  (74,227) -- (74,72) -- (314,72) -- (314,227);
			\draw (182,73.4) node [anchor=north west][inner sep=0.75pt]  [xscale=0.9,yscale=0.9]  {$1$};
			\draw (443.5,210.4) node [anchor=north west][inner sep=0.75pt]  [xscale=0.9,yscale=0.9]  {$x $};
			\draw (248.5,210.4) node [anchor=north west][inner sep=0.75pt]  [xscale=0.9,yscale=0.9]  {$\frac{1}{2} $};
			\draw (303,210.4) node [anchor=north west][inner sep=0.75pt]  [xscale=0.9,yscale=0.9]  {$1$};
			\draw (337,129.4) node [anchor=north west][inner sep=0.75pt]  [xscale=0.9,yscale=0.9]  {$F$};
			\draw (254,198.4) circle (2.5pt);
			\filldraw[black] (254,79.4) circle (2.5pt);
			\draw (306,198.4) circle (2.5pt);
			\filldraw[black] (306,79.4) circle (2.5pt);
			\draw (273,84.4) node [anchor=north west][inner sep=0.75pt]  [xscale=0.9,yscale=0.9]  {$G$};
			\draw (140,162.4) node [anchor=north west][inner sep=0.75pt]  [xscale=0.9,yscale=0.9]  {$\rho _{1}$};
			\draw (48,140.4) node [anchor=north west][inner sep=0.75pt]  [xscale=0.9,yscale=0.9]  {$\rho _{2}$};
		\end{tikzpicture}
		\caption{\red{This figure illustrates how the hypo-distance calculation (denoted by $\dl_\rho$) can vary depending on the chosen "box size" parameter (represented by $\rho$). We consider two boxes with $0<\rho_1<\rho_2$.  Let $F$ and $G$ be delta functions at points 1 and  $\frac{1}{2}$, respectively. Within the smaller box ($\rho_1$), the hypo-distance between $F$ and $G$ is zero (denoted by $\dl_{\rho}(F,G) = 0$). However, when we consider the larger box ($\rho_2$)), both the hypo-distance ($\dl$) and the $\rho$-distance ($\dl_\rho$) become strictly positive.}}
		\label{fig:remark2.13}
	\end{figure}
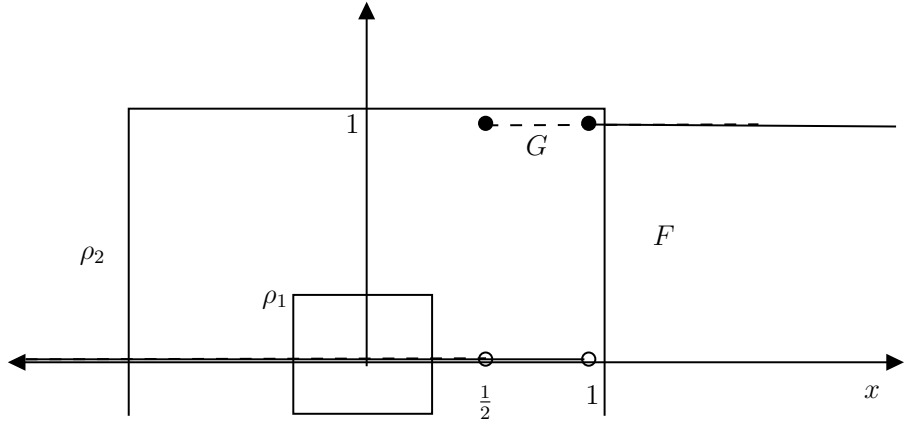
	
	Consider rectangles of sizes $0< \rho\leq\frac{1}{4}$.  
	We have that $\dl_{\rho}(F,G) = 0$, attributable to both $F$ and $G$ being equal to 0 over this rectangle.  For $\frac{1}{4} < \rho\leq\frac{1}{2}$, $\dl_{\rho}(F,G) = 2\rho-\frac{1}{2}$, and for $\frac{1}{2}< \rho$, $\dl_{\rho}(F,G) = \frac{1}{2}$.   However, $\dl(F,G) = (1/2)e^{-1}-2e^{-1/2}+2e^{-1/4}$. (This comes from computing the formula and the preceding calculations.)
\end{remark}

We end this section by describing the space $\cdfcns(S)$, and study its properties as a subset of $({\uscfcns}_+(S),\dl)$.

\subsection{The set of cumulative distribution functions}\label{ssec:cdf}

Recall from probability theory that every probability measure $P$ on the measurable space $(S, \mathcal{B}_S)$, with $\mathcal{B}_S$ the Borel $\sigma$-algebra over $S$, defines a cdf $F: S \to [0, 1]$ through $F(x) = P(B_x)$, for $x \in S$, where $B_x := \{y \in S\,\mid\, y \le x\}$.  Here, vector inequalities are understood component-wise.  Let us define the concept of cdfs.

\begin{definition}[Cumulative distribution function]\label{def:cdf}
	\label{cdf_def}
	Let $F:S\to[0,1]$ be a function defined over rectangular $S=[\alpha_1,\beta_1]\times\cdots\times[\alpha_m,\beta_m]$.  We say that $F$ is a cumulative distribution function (cdf) if $F\in\uscfcns_{+}(S)$ and it also satisfies
	\begin{enumerate}[label=(\roman*)]
		\item $F\left(x^\nu\right) \rightarrow 0$ if there is a component $i$ such that $x_i^\nu\to \alpha_i$,
		\item $F\left(x^{\nu}\right) \rightarrow 1$ if $x^{\nu} \rightarrow (\beta_1,\ldots,\beta_m)$; and
		\item the {\bf distribution condition}:  For every rectangle $A\subset S$, determined by its vertices $\{v^1,\ldots,v^{2^{m}}\}$\footnote{A rectangle $A$ defined by its vertices $\{v^1,\ldots,v^{n}\}$, can be described as $A=\mathop{\rm con}(\{v^1,\ldots,v^{n}\})$, i.e., the convex hull of $\{v^1,\ldots,v^{n}\}$.}, $F$ satisfies 
		\begin{equation}
			\label{eq:distcond}\Delta_{A} F:=\sum_{j=1}^{2^{m}}\left(\operatorname{sgn}_{A} v^{j}\right) F\left(v^{j}\right)\geq 0,
		\end{equation}
		where $\operatorname{sgn}_{A} v^{j}=1$ if the number of components $v_{i}^{j}$ at a lower bound of $A$ is even and $\operatorname{sgn}_{A} v^{j}=-1$ if the number is odd.
	\end{enumerate}
\end{definition}

We denote by $\mathcal{M}$ the set of all probability measures on $(S,\mathcal{B}_{S})$ and  denote the set of corresponding cdfs by
$$
\cdfcns(S):=\left\{F: S \rightarrow[0,1]\mid \exists P \in \mathcal{M} \text { with } F(x)=P\left(B_{x}\right), \forall x \in S\right\}.
$$
This set of functions is convex (in a pointwise manner) and it is also a subset of $\uscfcns(S)$. Furthermore, the hypo-distance between elements of this space is still bounded by 1, $\dl(f,g)\leq 1$ for any functions $f,g  \in \cdfcns(S)$, by virtue of Proposition~\ref{prop:less1}.

\red{Note that for a function $F\in\cdfcns_{+}(S)$, there exists an associated unique measure $\mu_F$ over $S$ such that for every rectangle $A\subset S$, $\mu_F(A)=\Delta_A\,F$ \cite[Thm.12.5]{billingsley1995}.  Thus, we will explore connections between cdfs and probability measures from a topological perspective.}

Convergence of probability measures is often guaranteed by compactness criteria.  One such is the \emph{tightness} condition, where a family of distribution measures cannot escape to infinity.  Let us recall its definition:

\begin{definition}[Tightness]
	A subset $C \subset\cdfcns(S)$ is {\bf tight} if for all $\varepsilon>0$ there exists a rectangle $A\subset S$ such that $\Delta_{A} F \geq 1-\varepsilon$, for every function $F \in C$.
\end{definition}

\red{The concept of tightness introduced here for families of functions in $\cdfcns(S)$ coincides with the classical definition of tightness for probability measures. In probability theory, a family of measures $\{\mu_\alpha\}_{\alpha\in\Lambda}$ on the real numbers is tight if for every $\varepsilon>0$ there exists a bounded rectangle $A$ such that $\sup_{\alpha\in\Lambda} \mu_\alpha(A)>1-\varepsilon$. 
	
	We can leverage this equivalence by associating each $F$ in a subset $C \subset\cdfcns(S)$  with a measure $\mu_F$.  This measure is defined over bounded rectangles $A$ as $\mu_F(A)=\Delta_A F$.  Consequently,  tightness of the set $C$ is equivalent to tightness of the family of measures $\{\mu_F:F\in C\}$.}

The connection between tightness and sequential convergence of measures is covered in \cite{billingsley1995}.  In particular, in Polish spaces, tightness is shown to be equivalent to pre-compactness in the topology of weak convergence, as demonstrated by Prokhorov's theorem \cite{Prokhorov}.

\begin{remark}[A topological remark from \cite{royset2017variational}]\label{rem:tight}  From \cite[Cor.~4.43]{rockafellar2009variational}, \red{in the case when $S$ is a nonempty closed subset of $\reals^m$,} we know that for a function $F \in \cdfcns(S)$ and for $r>0$, the ball centered at $F$ of radius $r$,
	\[\ball_{\dl}(F,r):= \{G \in \uscfcns(S) \ | \ \dl(F,G)\leq r \},\]
	is a compact set of $(\uscfcns(S),\dl)$.  However, the subset $\ball_{\dl}(F, r) \cap \cdfcns(S)$ is neither closed nor tight unless $r=0$.
	
	To see this, define the function $g: S \rightarrow[0,1]$ as $g(x)= \max \{0, F(x)-r\}\text{ and }r>0.$
	
	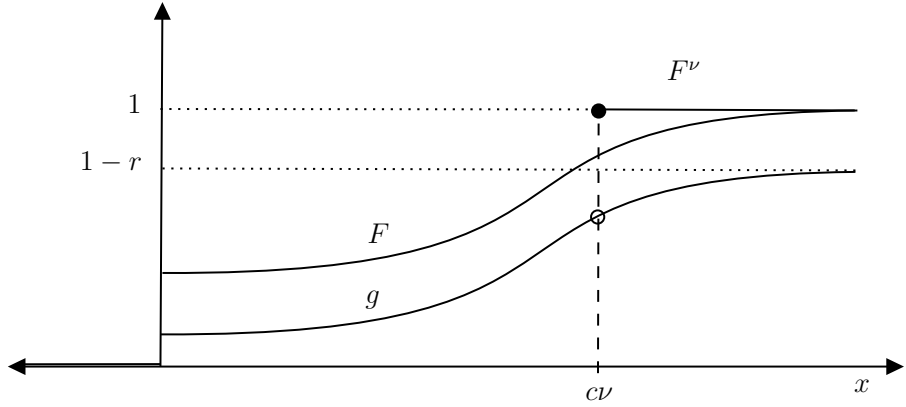
\begin{figure}[ht!]
		\centering
		\tikzset{every picture/.style={line width=0.75pt}} 
		\begin{tikzpicture}[x=0.75pt,y=0.75pt,yscale=-1,xscale=1]
			\draw    (97,190) -- (543,190) ;
			\draw [shift={(546,190)}, rotate = 180] [fill={rgb, 255:red, 0; green, 0; blue, 0 }  ][line width=0.08]  [draw opacity=0] (8.93,-4.29) -- (0,0) -- (8.93,4.29) -- cycle    ;
			\draw [shift={(94,190)}, rotate = 0] [fill={rgb, 255:red, 0; green, 0; blue, 0 }  ][line width=0.08]  [draw opacity=0] (8.93,-4.29) -- (0,0) -- (8.93,4.29) -- cycle    ;
			\draw    (171.98,9) -- (171,190) ;
			\draw [shift={(172,6)}, rotate = 90.31] [fill={rgb, 255:red, 0; green, 0; blue, 0 }  ][line width=0.08]  [draw opacity=0] (8.93,-4.29) -- (0,0) -- (8.93,4.29) -- cycle    ;
			\draw  [dash pattern={on 0.84pt off 2.51pt}]  (172,90) -- (522,91) ;
			\draw  [dash pattern={on 4.5pt off 4.5pt}]  (391.67,193.33) -- (392,60) ;
			\draw    (171,173.75) .. controls (420.67,174.83) and (308,94) .. (521.5,91.75) ;
			\draw  [dash pattern={on 0.84pt off 2.51pt}]  (171,60) -- (394.8,60.2) ;
			\draw    (101,188.83) -- (171,188.75) ;
			\draw    (392.8,60.2) -- (521.14,60.71) ;
			\draw    (172,142.75) .. controls (421.67,143.83) and (309,63) .. (522.5,60.75) ;
			\draw (129,80.4) node [anchor=north west][inner sep=0.75pt]  [xscale=0.9,yscale=0.9]  {$1-r$};
			\draw (384,199.4) node [anchor=north west][inner sep=0.75pt]  [xscale=0.9,yscale=0.9]  {$c\nu $};
			\draw (153,51.4) node [anchor=north west][inner sep=0.75pt]  [xscale=0.9,yscale=0.9]  {$1$};
			\draw (273,150.4) node [anchor=north west][inner sep=0.75pt]  [xscale=0.9,yscale=0.9]  {$g$};
			\draw (425,34.4) node [anchor=north west][inner sep=0.75pt]  [xscale=0.9,yscale=0.9]  {$F^{\nu }$};
			\draw (519.5,195.4) node [anchor=north west][inner sep=0.75pt]  [xscale=0.9,yscale=0.9]  {$x $};
			\draw (274,116.4) node [anchor=north west][inner sep=0.75pt]  [xscale=0.9,yscale=0.9]  {$F$};
			\draw (391.5,114.5) circle (2.5pt);
			\filldraw[black] (392,61.0) circle (2.5pt);
		\end{tikzpicture}
		\caption{From Remark~\ref{rem:tight}, a sequence of cdfs contained in a ball of radius $r$, that is nor tight neither convergent to a distribution function}
		\label{fig:tight}
	\end{figure}
	
	Then, it is easy to see that $g\in\ball_{\dl}(F,r)$ (i.e., $\dl(g, F) \leq r$). Consider a rectangle $A \subset S$, such that $\Delta_{A}(F) \geq 1-r$, and set $c>0$ such that $A \subset\left\{x \in S: x \leq c \mathbf{1}\right\}$.  Construct the sequence  $\{F^{\nu}\}_{\nu\in\nats}\subset \cdfcns(S)$ as follows 
	\begin{align*}
		F^{\nu}(x) = \left\{\begin{array}{cl}
			1 &\text{ if}\  c \nu \mathbf{1} \leq x\\
			g(x) &\text{ otherwise},\\
		\end{array}\right.
	\end{align*}
	as depicted in Figure~\ref{fig:tight}.  Since $F(x) \geq 1-r$ for $x$ with $F^{\nu}(x)=1$, we have that $\left|F^{\nu}(x)-F(x)\right| \leq r$ for all $x\in S$ and thus $\dl\left(F^{\nu}, F\right) \leq r$, i.e., $\{F^\nu\}_{\nu\in\nats}\subset\ball_{\dl}(F,r)$.   However, $\left\{F^{\nu}\right\}_{\nu\in\nats}$ is not tight and does not tend to a distribution function.
	
	The only balls of $\left(\uscfcns(S), \dl\right)$ contained in $\cdfcns(S)$ are those with zero radius, i.e., $\ball_{\dl}(F, 0)$. We observe that a setup centered on the metric space $\left(\cdfcns(S), \dl\right)$, instead of $\left(\uscfcns(S), \dl\right)$, is possible but has the disadvantage that the space is not complete.
\end{remark}

\begin{example}[Closure of the distribution condition under hypo-convergence]\label{ex:cdfclosure}
	Let $S$ be a compact rectangular subset of $\reals^2$.  If a sequence $\{F^\nu\}_{\nu\in\nats}\subset \uscfcns_{+}(S)$ such that for each $\nu$, $F^\nu$ satisfy the distribution condition (Def.~\ref{def:cdf}-\eqref{eq:distcond}), and hypo-converges to some function $F$, then it is not true that the limit satisfies the distribution condition, i.e.,  $\Delta_A F\geq0$ for every rectangle $A\subset S$. To see this, consider a rectangle $S\subset \reals^2$, and let $A=[x_1,x_2]\times [y_1,y_2]\subset S$ be a rectangle on the upper right corner of the domain depicted by Figure~\ref{fig:cdfclosure}.
	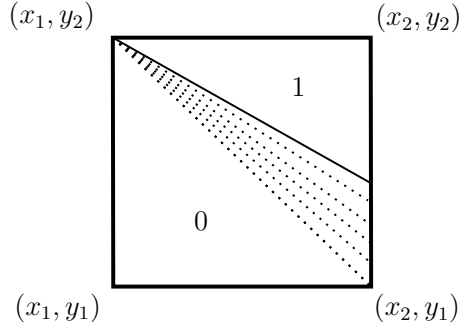
\begin{figure}[ht!]
		\centering
		\tikzset{every picture/.style={line width=0.75pt}} 		
		\begin{tikzpicture}[x=0.75pt,y=0.75pt,yscale=-1,xscale=1]
			\draw  [line width=1.5]  (250.28,19.82) -- (380.34,19.82) -- (380.56,145.45) -- (250.5,145.45) -- cycle ;
			\draw [line width=0.75]  [dash pattern={on 0.84pt off 2.51pt}]  (250.28,19.82) -- (380.56,145.45) ;
			\draw [line width=0.75]    (250.28,19.82) -- (380,93) ;
			\draw [line width=0.75]  [dash pattern={on 0.84pt off 2.51pt}]  (250.28,19.82) -- (379,102) ;
			\draw [line width=0.75]  [dash pattern={on 0.84pt off 2.51pt}]  (250.28,19.82) -- (380,114) ;
			\draw [line width=0.75]  [dash pattern={on 0.84pt off 2.51pt}]  (250.28,19.82) -- (381,124) ;
			\draw [line width=0.75]  [dash pattern={on 0.84pt off 2.51pt}]  (250.28,19.82) -- (381,134) ;
			\draw [line width=0.75]  [dash pattern={on 0.84pt off 2.51pt}]  (250.28,19.82) -- (380.56,145.45) ;
			\draw (199.35,147.53) node [anchor=north west][inner sep=0.75pt]  [xscale=0.9,yscale=0.9]  {$( x_{1} ,y_{1}) \ $};
			\draw (339,38.4) node [anchor=north west][inner sep=0.75pt]  [xscale=0.9,yscale=0.9]  {$1$};
			\draw (284,103.4) node [anchor=north west][inner sep=0.75pt]  [xscale=0.9,yscale=0.9]  {$ \begin{array}{l}
					0\\
				\end{array}$};
			\draw (381.35,1.4) node [anchor=north west][inner sep=0.75pt]  [xscale=0.9,yscale=0.9]  {$( x_{2} ,y_{2}) \ $};
			\draw (196.35,0.4) node [anchor=north west][inner sep=0.75pt]  [xscale=0.9,yscale=0.9]  {$( x_{1} ,y_{2}) \ $};
			\draw (380.56,147.85) node [anchor=north west][inner sep=0.75pt]  [xscale=0.9,yscale=0.9]  {$( x_{2} ,y_{1}) \ $};
		\end{tikzpicture}
		\caption{From Example~\ref{ex:cdfclosure}, the dotted lines represent (the level curves of) $\{F^\nu\}_{\nu\in\nats}$, with each $F^\nu$ satisfying the distribution condition. However, while this condition holds for every $F^\nu$, it fails to hold for its hypo-limit.}
		\label{fig:cdfclosure}
	\end{figure}
	
	Let $\{F^\nu\}_{\nu\in\nats}$ be a sequence such that $F^\nu(z_1,z_2) = 1$ if $z_1\geq x_1$ and $ z_2 \geq (1/\nu)y_2+(1-1/\nu)y_1$, and $0$ otherwise.
	This sequence is nondecreasing, usc, and also satisfies the distribution condition over $A$:
	$$\Delta_A F^\nu = F^\nu(x_2,y_2)-F^\nu(x_1,y_2)+F^\nu(x_1,y_1)-F^\nu(x_2,y_1) = 1-1-0+0 = 0.$$
	Also, the sequence hypo-converges to a function $F$ that is $0$ everywhere except in the upper triangle. We have that $F(x_2,y_1) = 1$, hence $\Delta_A F = -1$. This proves that the subset of $\uscfcns_+(S)$ satisfying the distribution condition is not closed in the hypo-distance topology.
\end{example}

We are interested in establishing some closedness property in $({\uscfcns}_+(S),\dl)$ with respect to the family of cdfs.  As tightness is not enough to guarantee that the hypo-limit of cdfs is a cdf as well, we propose to study equi-usc; see \cite{dolecki1983equiusc}.

\begin{definition}[equi-usc]
	The sequence $\left\{f^{\nu}:S\to[0,1]\right\}_{\nu \in \nats}$ is equi-usc at $\bar{x}\in S$ if for every $\varepsilon>0$ there exists $\delta>0$ with
	$$
	f^{\nu}(x) \leq f^{\nu}(\bar{x})+\varepsilon, \text { for all } \nu \in \nats \text { when } x \in S,\|x-\bar{x}\|_{\infty} \leq \delta.
	$$
\end{definition}

\begin{proposition}
	\label{prop:close-limits}
	Let $\{F^{\nu}\}_{\nu\in\nats}$ be a sequence in $\uscfcns(S)$ that is equi-usc at $(\alpha_1,\ldots,\alpha_m)$, and satisfies for all $\nu\in\nats$:
	\begin{align}
		&\lim_{k\to\infty} F^\nu(x^k) = 0,\,\text{whenever}\,\,x_i^k\to\alpha_i\,\,\text{for } i=1,\ldots,m, \label{limitcdf2}\\
		&\lim_{x \to (\beta_1,\ldots,\beta_m)}F^\nu(x) = 1. \label{limitcdf1}
	\end{align}
	If $F^\nu \hto F$ for some $F \in \uscfcns(S)$, then $F$ also satisfy \eqref{limitcdf2}-\eqref{limitcdf1}.
\end{proposition}

\begin{proof}
	Let $\varepsilon>0$.  By virtue of  \cite[Thm.\,7.10]{rockafellar2009variational}, the equi-usc of the sequence $\{F^{\nu}\}_{\nu\in\nats}$ at $(\alpha_1,\ldots,\alpha_m)$ implies its pointwise convergence, i.e., $F^\nu(\alpha_1,\ldots, \alpha_m) \to F(\alpha_1,\ldots, \alpha_m)$. Thus $F(\alpha_1,\ldots, \alpha_m) = 0$.  Finally, for fixed $\nu\in\nats$, there exists $\delta>0$ for which
	\[F^\nu(x) > 1-\varepsilon\]
	for all $x$ that $\|x-(\beta_1,\ldots,\beta_m)\|_\infty<\delta$.  The hypo-convergence implies that
	\[1 \geq F(x)\geq\limsup_{\nu\to\infty}F^\nu(x) \red{\geq} 1-\varepsilon,\]
	which proves \eqref{limitcdf1} after letting $\varepsilon\to 0$.
\end{proof}

\begin{example}\label{ex:equicex}
	We cannot omit the hypothesis of equi-usc at the point $(\alpha_1,\ldots,\alpha_m)$ in Proposition~\ref{prop:close-limits}. 
	Define 
	\[F^{\nu}(x) =\begin{cases}
		0 &\text{ if}\  x_i \in \left[\alpha_i, \alpha_i+\frac{|\beta_i-\alpha_i|}{\nu}\right),\, \forall i = 1,\ldots,m\\
		1 &\text{ otherwise}.
	\end{cases}\]
	Then, $F^\nu \hto F$, with $F$ being the function that has $F(x) = 1$ for every $x$ with $x_i \neq \alpha_i$ for all $i$, and has $F(\alpha_1,\ldots,\alpha_m) = 0$ if and only if $\lim_{\nu}F^\nu(\alpha_1,\ldots,\alpha_m) = F(\alpha_1,\ldots,\alpha_m)$. Figure~\ref{fig:equicex} illustrates this phenomena by means of a similar one-dimensional setting.
	
	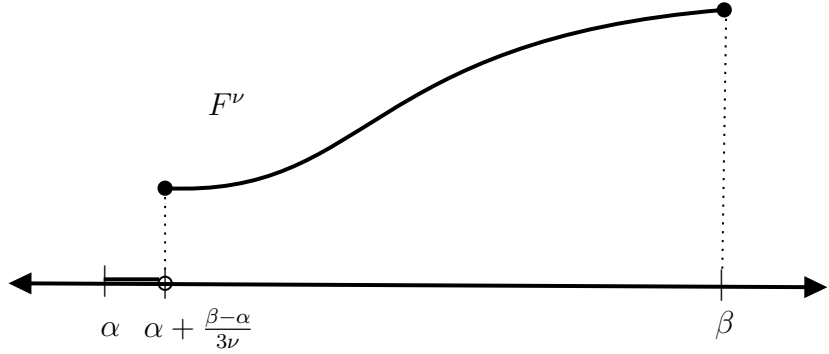
\begin{figure}[ht!]
		\centering
		\tikzset{every picture/.style={line width=0.75pt}} 		
		\begin{tikzpicture}[x=0.75pt,y=0.75pt,yscale=-1,xscale=1]
			\draw [line width=1.5]    (144,155.02) -- (549,156.98) ;
			\draw [shift={(553,157)}, rotate = 180.28] [fill={rgb, 255:red, 0; green, 0; blue, 0 }  ][line width=0.08]  [draw opacity=0] (11.61,-5.58) -- (0,0) -- (11.61,5.58) -- cycle    ;
			\draw [shift={(140,155)}, rotate = 0.28] [fill={rgb, 255:red, 0; green, 0; blue, 0 }  ][line width=0.08]  [draw opacity=0] (11.61,-5.58) -- (0,0) -- (11.61,5.58) -- cycle    ;
			\draw [line width=1.5]    (219,107) .. controls (323,112) and (317,32) .. (501,17) ;
			\draw  [dash pattern={on 0.84pt off 2.51pt}]  (502,17) -- (500,150) ;
			\draw  [dash pattern={on 0.84pt off 2.51pt}]  (219,107) -- (219,148) ;
			\draw [line width=1.5]    (188,153) -- (216,153) ;
			\draw (185,173.4) node [anchor=north west][inner sep=0.75pt]    {$\alpha$};
			\draw (495,167.4) node [anchor=north west][inner sep=0.75pt]    {$\beta$};
			\draw (185,144.4) node [anchor=north west][inner sep=0.75pt]    {$|$};
			\draw (215.43,145.4) node [anchor=north west][inner sep=0.75pt]    {$|$};
			\draw (496,146.4) node [anchor=north west][inner sep=0.75pt]    {$|$};
			\filldraw[black] (219,107) circle (2.5pt);
			\draw (239,57.4) node [anchor=north west][inner sep=0.75pt]    {$F^{\nu }$};
			\filldraw[black] (501,17) circle (2.5pt);
			\draw (206.71,166.26) node [anchor=north west][inner sep=0.75pt]    {$\alpha+\frac{\beta-\alpha}{3\nu }$};
			\draw (219,155) circle (2.5pt);
		\end{tikzpicture}
		\caption{One-dimensional instance of Example~\ref{ex:equicex}, where the concept of equi-usc at $\alpha$ becomes crucial in ensuring the validity of the limit \eqref{limitcdf2}.}
		\label{fig:equicex}
	\end{figure}	
\end{example}

\begin{corollary}
	\label{cor:dist-condition-closed}
	Let $A$ be any rectangle in $S$ with vertices $\{v_j\}_{j=1}^{2^m}$, and denote by $A^-$ the set of indexes such that $\operatorname{sign}_A v_j = -1$ and $A^+$ the ones with positive sign.  Consider a sequence $\{F^{\nu}\}_{\nu\in\nats}\subset\uscfcns(S)$ that is equi-usc at every vertex of $A^-$ and \eqref{eq:distcond} holds for every function of the sequence, i.e., $\Delta_A F^\nu \geq 0$, for all $\nu\in\nats$. \red{If $F^\nu$ hypo-converge to $F$, then $F$ satisfies} the distribution condition in \eqref{eq:distcond}, i.e., $\Delta_A F\geq 0$.  
\end{corollary}

\begin{proof}
	The property $\Delta_A F^\nu \geq 0$ can be written as
	\[\sum_{j \in A^{+}}F^\nu(v_j)\geq \sum_{j \in A^{-}}F^\nu(v_j),\]
	\red{and we would like to show that this remains true for $F$. From the hypo-convergence of $\{F^\nu\}$, 
		\[\sum_{j \in A^{+}}F^\nu(v_j)\geq \sum_{j \in A^{-}}F(v_j),\]
		and on the other hand,}
	\[ \begin{aligned}\sum_{j \in A^{+}}F(v_j)\geq \limsup_\nu\sum_{j \in A^{+}}F^\nu(v_j)&\geq \limsup_\nu \sum_{j \in A^{-}}F^\nu(v_j) \\
		&\geq \liminf_\nu \sum_{j \in A^{-}}F^\nu(v_j)\geq \sum_{j \in A^{-}}\liminf_\nu F^\nu(v_j).
	\end{aligned}\]
	Additionally, the equi-usc at every $v_j \in A^-$ implies that $\liminf_\nu F^\nu(v_j) = F(v_j)$ which concludes the proof. 
\end{proof}

\subsection{Epi-splines}

In this section, we adapt the theory of epi-splines developed in the context of \red{extended real-valued} lsc functions in \cite{royset2016multivariate,royset2016erratum} where they are defined as piecewise polynomial functions that can approximate any lsc function to an arbitrary level of accuracy. The following Proposition~\ref{prop:epi-usc} and Theorem~\ref{thm:denseapproximation} are adaptations of  \cite[Prop.\,3.3]{royset2016multivariate,royset2016erratum}  and \cite[Thm.\,3.5]{royset2016multivariate,royset2016erratum}  for  usc functions which allow us to reformulate an approximating, finite-dimensional estimation problem.  We continue using the name \textit{epi-spline} even though the definition is re-oriented towards usc functions.  \red{The adapted definition epi-splines \cite[Def.\,3.2]{royset2016multivariate,royset2016erratum} to the $\uscfcns(S)$ setting can be stated as}:

\begin{definition}[Epi-spline] 
	An epi-spline (in the \red{usc} sense) $s :S \rightarrow \red{[0,1]}$ of order $p \in \nats_0$ with box partition
	$\mathcal{R}=\left\{R_{k}\right\}_{k=1}^{N}$ of $S$, is a function that
	\begin{enumerate}
		\item on each $R_{k}, k=1, \ldots, N$, is polynomial of total degree $p$,
		\item and for every $x \in S$, has $s(x)=\red{\limsup_{x^{\prime} \rightarrow x} s\left(x^{\prime}\right)}$.
	\end{enumerate}
	We denote the family of epi-splines (in the usc sense) of order $p\in\nats_{0}$ by $\episplns^p(\mathcal{R})$.   We omit the phrase `in the usc sense' in the following for brevity.
\end{definition}

\begin{proposition}\label{prop:epi-usc}
	For any  partition $\mathcal{R}$ of $S$, and $p \in \nats_{0}$,
	$$\episplns^p(\mathcal{R}) \subset\uscfcns(S).$$
\end{proposition}

Next, we adapt the density result in \cite[Thm.\,3.5]{royset2016multivariate,royset2016erratum} to the usc case:

\begin{theorem}[Dense approximation]
	\label{thm:denseapproximation}
	For any $p \in \nats_{0}$ and $\left\{\mathcal{R}^{\nu}\right\}_{\nu=1}^{\infty}$, an infinite refinement of $S$,
	$$
	\bigcup_{\nu=1}^{\infty} \episplns^p\left(\mathcal{R}^{\nu}\right)  \text{ is dense in } \uscfcns(S).
	$$
\end{theorem}

\begin{proof}
	Let $f \in \uscfcns(S)$, $\{\mathcal{R}^\nu\}_{\nu=1}^{\infty}$ an infinite refinement of $S$, and $\mathcal{R}^{\nu}=\left\{R_{k}^{\nu}\right\}_{k=1}^{N^{\nu}}$ a partition of $S$. It suffices to construct a sequence of epi-splines of order $p=0$. For every $\nu \in \nats$ and $R_{k}^{\nu}, k=1,2, \ldots, N^{\nu}$. We define
	$$\sigma\left(R_{k}^{\nu}\right)= \sup _{x \in \mathrm{cl} R_{k}^{\nu}} f(x)$$
	and construct $s^{\nu}: S\rightarrow [0,1]$ as follows:
	$$
	s^{\nu}(x) = \max _{k=1,2, \ldots, N^{\nu}}\left\{\sigma\left(R_{k}^{\nu}\right) \ \mid \ x \in \operatorname{cl}R_k^\nu \right\}, 
	$$
	Clearly, $s^{\nu}$ is constant on each $R_{k}^{\nu}, k=1,2, \ldots, N^{\nu}$ and satisfies 
	$$\limsup _{x^{\prime} \rightarrow x} s^{\nu}\left(x^{\prime}\right)= s^{\nu}(x) \text{ for all } \ x \in S.$$
	
	Hence, $s^{\nu} \in \episplns^{0}\left(\mathcal{R}^{\nu}\right)$ and consequently also in $\episplns^{p}\left(\mathcal{R}^{\nu}\right)$ for $p \in \nats$. We next show that the two conditions of \cite[Prop.\,2.1]{royset2017variational} holds. Let $x \in S$ be arbitrary. By upper semicontinuity of $f$, for every $\varepsilon>0$ there exists $\delta>0$ such that
	\[f\left(x^{\prime}\right) \leq f(x)+\varepsilon \text{ whenever } x^{\prime} \in \ball(x, \delta).\] 
	
	Since $\left\{\mathcal{R}^{\nu}\right\}_{\nu=1}^{\infty}$ is an infinite refinement, there also exist $\bar{\nu}$ and $\gamma\in(0,\delta)$ such that $R_{k}^{\nu} \subset \ball(x, \gamma)$ for every $\nu \geq \bar{\nu}$ and $k$ satisfying $R_{k}^{\nu}\cap\ball(x,\gamma)\neq \emptyset$. Hence, for $x^\prime\in\ball(x,\gamma/2)$, 
	\[s^{\nu}(x^\prime) \leq \min_k \sup_{z \in \operatorname{cl}\mathcal{R}_k^\nu}\left\{\left.f(z)\,\right|\, x^\prime \in {\rm cl}\,\mathcal{R}_k^\nu\right\}\leq f(x)+\varepsilon,\,\text{for\,all}\,\nu\geq \bar{\nu}.\]
	Thus, for every sequence $x^{\nu} \rightarrow x$,
	$$
	\limsup _{\nu} s^{\nu}\left(x^{\nu}\right) \leq f(x)+\varepsilon.
	$$
	Since $\varepsilon$ is arbitrary, $\limsup_\nu s^{\nu}\left(x^{\nu}\right) \leq f(x)$ and condition a) of \cite[Prop.\,2.1]{royset2017variational} holds. As $f(x)\in[0,1]$ for $x \in S$, we have that $s^\nu(x)\geq f(x)$ and $s^\nu(x)\geq -\nu$ for $\nu$ sufficiently large, so
	\begin{align*}
		s^{\nu}(x) \geq f(x),\,\text{  for all }\nu\text{ sufficiently large.}
	\end{align*} 
	It follows that
	\begin{align*}
		\liminf _{\nu} s^{\nu}\left(x^{\nu}\right)=\liminf _{\nu} s^{\nu}(x) \geq f(x).
	\end{align*}
\end{proof}

By considering only rational epi-splines of $\episplns^{0}\left(\mathcal{R}^{\nu}\right)$ in the proof of Theorem~\ref{thm:denseapproximation}, i.e., functions $s: S \rightarrow [0,1]$ with $s(x)=q_{k}$ for $x \in R_{k}^{\nu}$ and $q_{k}$ a rational constant, $k=1, \ldots, N^{\nu}$. Specifically, in that proof one can replace $\sigma\left(R_{k}\right)=\sup_{x \in \operatorname{cl} R_{k}^{\nu}} f(x)$ by $\sigma\left(R_{k}^{\nu}\right)$ equals any rational number in
$$\left[ \max _{x \in \operatorname{cl} R_{k}^{\nu}} f(x),\max _{x \in \operatorname{cl} R_{k}^{\nu}} f(x)+ 1 / \nu\right]$$
and the next result follows.

\begin{corollary}
	\label{separability-usc}
	For $p \in \nats_{0}$ and $\left\{\mathcal{R}^{\nu}\right\}_{\nu=1}^{\infty}$, an infinite refinement of $S$, $\left(\uscfcns(S), \dl\right)$ is separable, with the rational epi-splines of $$\bigcup_{\nu=1}^{\infty} \episplns^{p}\left(\mathcal{R}^{\nu}\right)$$ furnishing a countable dense subset.
\end{corollary}

\red{The density and separability results stated in Theorem~\ref{thm:denseapproximation} and Corollary~\ref{separability-usc} extend to the space $\uscfcns_+(S)$ with the same approximating epi-spline construction.  A similar result holds for $\Lipfcns_\kappa(S)$ when $p\in\nats$. This is justified by the first-order epi-spline construction in \cite[Thm.\,3.8,\, Thm.\,3.11]{royset2016multivariate}, provided additional requirements like simplicial partitions are met. }

\section{Existence and approximation}\label{sec:thms}

In this section, we formulate our mathematical program tailored to address the constrained estimation problem under stochastic ambiguity. We present a sequence of discretized problems.  Further, we furnish a convergence proof justifying the discretization scheme. 

In view of the close connection between the hypo-distance and the hat-distance (cf. Section~\ref{sec:notandbg}), we proceed by adopting the hat-distance as a surrogate for the hypo-distance as it appears in constrained estimation problem under stochastic ambiguity.  \red{We also assume for the rest of the manuscript that $S$ is a compact set of $\reals^m$, and $\rho \geq 1$ is such that $S\subset \rho\ball$.  Under this setting, bounds on hat-distance and $\rho$-distance are more tractable (see \cite[Prop.\,4.37]{rockafellar2009variational}).}  This leads to the problem statement: 

\begin{problem}
	\label{problem2}
	For a rectangular $S$, let $F_0,G_0\in\cdfcns(S)$, let $\mathscr{F}:=\uscfcns_{+}(S)$, and let $\rho>0$ such that \red{$S\subset \rho\ball$}.  The problem is to find $\hat F:S\to[0,1]$ such that 
	\begin{equation*}
		\hat{F} \in \argmin_{F \in \mathscr{F}}\ \left\{\left.\hatdl_\rho(F, F_{0})\,\right|\,
		\hatdl_\rho(F,G_0)\leq \delta
		\right\}.
	\end{equation*}
\end{problem}
\red{We know that $(\uscfcns_+(S),\dl)$ is a complete metric space in which closed and bounded sets are compact \cite[Thm.\,4.42, Cor.\,4.43]{rockafellar2009variational}. This provides the existence of solutions for Problem~\ref{problem2}. Nevertheless, we cannot expect uniqueness of the solutions by definition of the hat-distance. }

\red{Moreover, $(\uscfcns_+(S),\dl)$ is separable. In particular, our approximation scheme relies on the discretization of the domain. Let $\{\mathcal{R}^\nu\}_{\nu=1}^\infty$ be an infinite refinement of $S$.  Then, following Corollary~\ref{separability-usc}, we can consider the rational epi-splines as the dense subset. We denote by $\mathscr{F}^\nu$ the projection of $\mathscr{F}$ onto $\episplns^{1}(\mathcal{R}^\nu)$, i.e. $\mathscr{F}^\nu = \mathscr{F} \cap \episplns^{1}(\mathcal{R}^\nu)$.}

Considering the approximations discussed in Section~\ref{sec:notandbg}, in terms of the space of usc functions and distance bounds for the hypo-distance, we define the following sequence of finite-dimensional optimization problems:\\

\begin{problem}
	\label{p_eta_rho_esplns_nu}
	For a rectangular $S$, let $F_0,G_0\in\cdfcns(S)$, let $\mathscr{F}^\nu := \mathscr{F} \cap \episplns^{1}(\mathcal{R}^\nu)$, and let $\rho>0$.  The problem is to find $\hat F:S\to[0,1]$ such that 
	\begin{equation*}
		\hat{F} \in \argmin_{F \in  \mathscr{F}^\nu}\ \left\{\left.\eta_{\rho}^{+,\nu}(F, F_{0})\,\right|\,
		\eta^{+,\nu}_{\rho}(F,G_0)\leq \delta
		\right\},
	\end{equation*}
	where $\eta_\rho^{+,\nu}(F,G_0)$ corresponds to the value of $\eta_\rho^+(F,G_0)$ over the particular box partition $\mathcal{R}^\nu$ of $S$.
\end{problem}

Note that we can define $\eta_\rho^{-,\nu}(F,G_0)$ analogously, as the value of $\eta_\rho^{-}(F,G_0)$ using the box partition $\mathcal{R}^{\nu}$ of $S$.

In order to analyze Problem~\ref{p_eta_rho_esplns_nu}, we define the following functions: $\varphi,\varphi^\nu: (\mathscr{F},\dl)\rightarrow \Reals$ where  $\varphi(F):=\hatdl_\rho(F,F_0) + \iota_{C}(F)$,  $\iota_C$ being the indicator function\footnote{$\iota_C(F) =0$ if $F \in C$ and takes the value $+\infty$ otherwise.} of the set $C:= \{ F\in\mathscr{F}\, \mid \ \hatdl_\rho(F,G_0)\leq \delta\}$, and $\varphi^\nu(F):=\eta_\rho^{+,\nu}(F,F_0)+ \iota_{C^\nu}(F) + \iota_{\mathscr{F}^{\nu}}(F)$, 
where 
\[C^\nu := \left\{ F \in \mathscr{F}\ \Bigg\vert
\   \begin{array}{c}
	\displaystyle \max_{y\in Y_\delta(l^{k})} F\left(l^{k}+y\right)+\delta \geq \min \left\{G_{0}\left(u^{k}\right), \rho\right\},\ \forall k=1,\ldots,N^\nu \\
	\displaystyle \max_{y\in Y_\delta(l^{k})} G_{0}\left(l^{k}+y\right)+\delta \geq \min \left\{F\left(u^{k}\right), \rho\right\},\  \forall k=1,\ldots,N^\nu
\end{array} \right\}.\]

The sequence of functions $\{\varphi^\nu\}_{\nu\in\nats}$ represents the objective function of approximating problems associated with the discretization of usc functions on the box partition $\mathcal{R}^\nu$, using the associated epi-splines.  It also approximates the $\rho$-distance by the upper bounds $\eta^{+,\nu}_\rho$.  Thus, it is the objective function representing Problem~\ref{p_eta_rho_esplns_nu}.  \red{Recall that for fixed $F_0$ and $G_0$, the continuity of $\eta_\rho^{+,\nu}(\cdot, F_0)$ and $\eta_\rho^{+,\nu}(\cdot, G_0)$ over $\mathscr{F}^\nu$, a finite-dimensional space, ensures the compactness of the feasible set and, consequently, guarantees the existence of solutions for Problem~\ref{p_eta_rho_esplns_nu}.}

In what follows, we analyze an approximation scheme for Problem~\ref{problem2} by considering the convergence of the functions representing the corresponding optimization problems. We provide conditions for the convergence of $\{\varphi^\nu\}_{\nu\in\nats}$ to $\varphi$, and we prove that cluster points of solutions to Problem~\ref{p_eta_rho_esplns_nu} are solutions to Problem~\ref{problem2}.

\begin{proposition}
	\label{prop1_lipschitz}
	For every $f \in \mathscr{F}$ and every sequence $\{f^\nu\}_{\nu\in\nats} \subset \mathscr{F}$ that hypo-converges to $f$, we have that 
	\begin{equation}
		\label{liminf_lipschitz}
		\liminf_{\nu}\varphi^{\nu}(f^\nu)\geq \varphi(f).   
	\end{equation}
\end{proposition}
\begin{proof}
	Let $\mathcal{R}^\nu$ be an infinite refinement of $S$.  First observe that by virtue of Theorem~\ref{thm:app-hatdistance} we have that 
	\begin{align*}
		\hatdl_\rho(f,g) \leq \eta^{+,\nu}_\rho(f,g), \quad f,g \in \mathscr{F},
	\end{align*}
	and so every limit of subsequences $f^\nu \in C^\nu$ is in $C$, thus $\LimOut C^\nu \subset C$.
	
	Let $f \in \mathscr{F}$ and $\{f^\nu\}_{\nu\in\nats} \subset \mathscr{F}$ that $f^\nu\hto f$. If $f \not\in C$ and $f^\nu\not\in \mathscr{F}^\nu$ then inequality~\eqref{liminf_lipschitz} holds and this is still true if $f \in C$ and $f^\nu\not\in \mathscr{F}^\nu$. If $f\in C$, $f^\nu\in\mathscr{F}^\nu$ but $f^\nu\not\in C^\nu$ then inequality~\eqref{liminf_lipschitz} holds. Let $f\in C$ and $f^\nu\in C^\nu \cap \mathscr{F}^\nu$ be any sequence that $f^\nu \hto f$ (this is equivalent to $\dl(f^\nu,f)\to 0$). 
	We have that 
	\[\displaystyle\hatdl_\rho(f^\nu,F_0)\leq  \eta_\rho^{+,\nu}(f^\nu,F_0)\]
	because of Theorem~\ref{thm:app-hatdistance}. On the other hand for every $\bar{\rho}>0$, $\dl(f^\nu, f)\to 0$ if and only if $\hatdl_\rho(f^\nu,f)\to 0$ for all $\rho\geq\bar{\rho}$, because of \cite[Thm.\,4.36]{rockafellar2009variational} and so
	\[\dl(f^\nu,F_0)\to \dl(f,F_0) \iff \hatdl_\rho(f^\nu,F_0)\to \hatdl_\rho(f,F_0) \text{ for all }\rho\geq\bar{\rho},\] 
	and then
	\[\liminf_{\nu}\varphi^\nu(f^\nu) = \liminf_{\nu}\left\{ \eta_\rho^{+,\nu}(f^\nu,F_0)\right\}\geq \lim_{\nu} \hatdl_\rho(f^\nu,F_0) = \hatdl_\rho(f,F_0) = \varphi(f).\]
	
	If $f\not\in C$ and $f^\nu \in \mathscr{F}^\nu$ then it does not exist any $f^\nu\in C^\nu$ that hypo-converges to $f$. Suppose there exists $f^\nu \in C^\nu$ that satisfies the condition. As $f\not\in C$, $\hatdl_\rho(f,G_0)>\delta$ and
	\[\liminf_{\nu\to\infty}\eta_\rho^{+,\nu}(f^\nu,G_0) \geq \liminf_{\nu\to\infty} \hatdl_\rho(f^\nu,G_0)= \hatdl_\rho(f,G_0)>\delta  \]
	and so there exists $\bar{\nu}$ such $\eta_\rho^{+,\nu}(f^\nu,G_0) > \delta$ for all $\nu\geq\bar{\nu}$.  This contradicts the fact that $f^\nu \in C^\nu$ for all $\nu$.
\end{proof}

Our objective is to find conditions that allow us to approximate solutions of Problem~\ref{problem2} by solving a sequence of Problem~\ref{p_eta_rho_esplns_nu}. To achieve this, we adopt the approach proposed in \cite[Ch.\,7E]{rockafellar2009variational}, wherein convergence in minimization is established through epi-convergence. Consider the following set of assumptions:

\begin{assumption}
	\label{ass:growing_assumption}
	For $\rho>0$, suppose that for every $f,F_0 \in \mathscr{F}$, and $\{f^\nu\}_{\nu\in\nats}$ hypo-converges to $f$.  Then, the following holds 
	\begin{equation}
		\limsup_{\nu}\eta_\rho^{+,\nu}(f^\nu,F_0)- \eta_\rho^{-,\nu}(f^\nu,F_0) \leq 0.
	\end{equation}
\end{assumption}

\begin{assumption}
	\label{ass:qualification_constraint}
	For $\rho>0$, suppose that for $f\in\mathscr{F}$ such that  $\hatdl_\rho(f,G_0) = \delta$, there exists a sequence $\{f^\nu\}_{\nu\in\nats} \subset \mathscr{F}$ such that $f^\nu \hto f$ and
	\begin{equation}
		\label{eq_assumption}
		\hatdl_\rho(f^\nu,G_0)< \delta, \quad \forall \nu \in\nats. 
	\end{equation}
\end{assumption}

We are now ready to state the first convergence results of our proposed approximation, which show that the sequence of objective functions epi-converges to the objective function of Problem~\ref{problem2}.  Additionally, Assumption~\ref{ass:qualification_constraint} is a commom assumption that relates to the Slater condition over the ambiguity set $\hatdl_\rho(f,G_0)\leq \delta$.

Recall that for functions $\{\varphi,\varphi^\nu\}_{\nu\in\nats}\subset(\mathscr{F},\dl)$, we say that $\{\varphi^\nu\}$ epi-converges to $\varphi$, denoted by $\varphi^\nu\eto \varphi$, if for every $f^\nu\hto f$, $\liminf_\nu \varphi^\nu(f^\nu)\geq \varphi(f)$, and for every $f\in \mathscr{F}$ there exists a sequence $f^\nu\hto f$ such that $\limsup_\nu \varphi^\nu(f^\nu)\leq \varphi(f)$.

\begin{theorem}\label{thm:epi}
	\red{Let $\rho\geq 1$ such that $S\subset\rho\ball$.  }Under Assumptions~\ref{ass:growing_assumption}~and~\ref{ass:qualification_constraint}, for every $f\in\mathscr{F}$ there exists a sequence of functions $f^\nu\in\mathscr{F}$ that \red{hypo-converges} to $f$ and 
	\begin{equation}
		\label{limpsup_lipschitz}
		\limsup_{\nu}\varphi^\nu(f^\nu)\leq \varphi(f). 
	\end{equation}
	Moreover $\varphi^\nu\eto \varphi$ and $\{C^\nu\}_{\nu\in\nats}$ set-converges to $C$.
\end{theorem}

\begin{proof}
	If $f \not\in C$, then $\varphi(f) = +\infty$ and \eqref{limpsup_lipschitz} holds. Let $f\in C$ and for now suppose there exists $\{f^\nu \in \mathscr{F}^\nu\}_{\nu\in\nats}$ that hypo-converges to $f$. 
	From Assumption \ref{ass:growing_assumption} we have that 
	\begin{equation}
		\label{ec:limsup}
		0 \leq \limsup_{\nu}\eta_\rho^{+,\nu}(f^\nu,F_0)- \eta_\rho^{-,\nu}(f^\nu,F_0) = 0.
	\end{equation}
	Thus, we have that
	\begin{align*}
		\limsup_\nu \eta_\rho^{+,\nu}(f^\nu,F_0) - \limsup_\nu\,& \eta_\rho^{-,\nu}(f^\nu,F_0)\\
		&=\limsup_\nu \eta_\rho^{+,\nu}(f^\nu,F_0)+\liminf_\nu -\eta_\rho^{-,\nu}(f^\nu,F_0)\\
		&\leq \limsup_\nu \eta_\rho^{+,\nu}(f^\nu,F_0) - \eta_\rho^{-,\nu}(f^\nu,F_0)\\
		&\leq 0.
	\end{align*}
	
	From Proposition~\ref{prop:est-rho-hypo-dist} we have that 
	\begin{equation*}
		\eta_\rho^{-,\nu}(f^\nu,F_0) \leq \hatdl_\rho(f^\nu,F_0) \leq  \eta_\rho^{+,\nu}(f^\nu,F_0).
	\end{equation*}
	
	Then
	\begin{equation*}
		\limsup_\nu  \eta_\rho^{-,\nu}(f^\nu,F_0)\leq \limsup_\nu  \eta_\rho^{+,\nu}(f^\nu,F_0).
	\end{equation*}
	
	Combining the inequalities above we have that $$\limsup_\nu \eta_\rho^{+,\nu}(f^\nu,F_0) = \limsup_\nu \eta_\rho^{-,\nu}(f^\nu,F_0)$$ and
	\begin{equation}
		\label{ec_limsup_lipschitz}
		\limsup_\nu \eta_\rho^{+,\nu}(f^\nu,F_0) = \limsup_\nu \hatdl_\rho(f^\nu,F_0) =\lim_\nu \hatdl_\rho(f^\nu,F_0) = \hatdl_\rho(f,F_0),
	\end{equation}
	where the last equality comes from the fact that $\dl(f^\nu,f)\to 0$ and  \cite[Thm.\,4.36]{rockafellar2009variational}. In particular,
	\[ \limsup_{\nu}\eta_\rho^{+,\nu}(f^\nu,F_0)\leq\hatdl_\rho(f,F_0).\]
	In order to have
	\[ \limsup_{\nu}\varphi^\nu(f^\nu)\leq \varphi(f),\]
	it suffices to show that $f^\nu \in C^\nu$. Then, we need to prove that for every $f\in C$ there exists a sequence $f^\nu \in C^\nu$ that $f^\nu \hto f$. The latter is equivalent to have that $C \subset \LimInn C^\nu$ so by proving this we will have \eqref{limpsup_lipschitz} and also the remaining condition for the set convergence of $C^\nu$ to $C$.
	
	For $f \in C$, there are two cases. The first case is that $\hatdl_\rho(f,G_0)<\delta$. By density of $\mathscr{F}^\nu$ in $\mathscr{F}$, there exists $f^\nu \in \mathscr{F}^\nu$ that hypo-converge to $f$, \red{and the fact that $S$ is a compact set and $S\subset\rho\ball$, using the triangle inequality for $\dl_\rho$, which coincides with $\hatdl_\rho$ (see the proof of \cite[Prop.\,7.49]{royset2021optimization}), we have that} 
	\[\hatdl_\rho(f^\nu, G_0) \leq \hatdl_\rho(f, G_0) + \hatdl_\rho(f^\nu, f)\quad \forall \nu.\]

	The hypo-convergence implies that for $\varepsilon = \delta - \hatdl_\rho(f, G_0) >0$ there exists $\bar{\nu}\in\nats$ that 
	\begin{align*}
		\hatdl_\rho(f^\nu, G_0) \leq \delta \quad \forall \nu\geq \bar{\nu}.
	\end{align*}
	
	The second case is that $\hatdl_\rho(f,G_0)=\delta$. By Assumption~\ref{ass:qualification_constraint}, there exists $\{f^n\}_{n\in\nats} \subset \mathscr{F}$ such that $f^n \hto f$ and satisfies \eqref{eq_assumption}. For fixed $n$ there exists $g^\nu_n \in \mathscr{F}^\nu$ that $g^\nu_n \hto f^n$ when $\nu\to\infty$. 
	Repeating the arguments for the first case we have that for some $\bar{\nu}_n$, 
	\[\hatdl_\rho(g^\nu_n,G_0) < \delta \quad \text{ for every }\nu \geq \bar{\nu}_n.\]
	
	We choose $f^\nu$ as $g^\nu_n$ for $\nu \geq \bar{\nu}_n$. In both cases we obtain a sequence that hypo-converges to $f$ when $\nu\to\infty$ and $\hatdl_\rho(f^\nu,G_0)\leq \delta$ for every $\nu$.  But we need $\eta_\rho^+(f^\nu,G_0)\leq \delta$. However in view that $f^\nu \hto f$, because of \eqref{ec_limsup_lipschitz}, we have that 
	\[ \limsup_{\nu} \eta^{+}_\rho(f^\nu,G_0) = \hatdl_\rho(f^\nu,G_0) \leq \delta,\]
	so for sufficiently large $\nu$, $\eta^{+}_\rho(f^\nu,G_0)\leq \delta$ and in consequence $f^\nu \in C^\nu$ for sufficiently large $\nu$.
\end{proof}

We establish the theoretical foundation for our proposed algorithm by proving that near minimizers of the approximated (finite dimensional) problem converge to a solution of the Problem~\ref{problem2}.  This result follows a standard argument, but we include a proof for completeness. 

\begin{theorem}
	\label{th:minimizers}
	Under Assumptions \ref{ass:growing_assumption} and \ref{ass:qualification_constraint}, every cluster point of sequences constructed from near minimizers of $\varphi^\nu$ in $\mathscr{F}$ is contained in $\operatorname{argmin}_{F \in \mathscr{F}}\varphi(F)$ provided that $\varepsilon^\nu$ vanishes, i.e, for $\varepsilon^\nu\to 0$, 
	$$ \operatorname{LimOut}\left(\varepsilon^\nu\operatorname{-argmin}_{F\in \mathscr{F}} \varphi^\nu(F)\right)\subset \operatorname{argmin}_{F \in \mathscr{F}}\varphi(F).$$
\end{theorem}
\begin{proof}
	Let
	\begin{align}
		f^\star \in \operatorname{LimOut}\left(\varepsilon^\nu\operatorname{-argmin}_{F\in \mathscr{F}} \varphi^\nu(F)\right).
	\end{align}
	Then, there exists $\{\nu_k, k\in\nats\}$ and $f^{k}\in \varepsilon^\nu\operatorname{-argmin}_{F\in \mathscr{F}^\nu} \varphi^\nu(F)$ such that $f^{k}\hto f^\star$. Let $g \in \operatorname{argmin}_{F\in\mathscr{F}}\varphi(F)$.  Under Assumptions~\ref{ass:growing_assumption}~and~\ref{ass:qualification_constraint}, by virtue of Theorem~\ref{thm:epi}, $\{\varphi^\nu\}_{\nu\in\nats}$ epi-converges to $\varphi$, and we have that there exists a sequence $\{g^\nu\}_{\nu\in\nats}$ in $\mathscr{F}$ such that $g^\nu\hto g$ and  
	\begin{align*}
		\limsup_{\nu} \varphi^\nu(g^\nu) \leq \varphi(g).
	\end{align*}
	Since $g \in C$ and $C^\nu \subset C$ we have that there exists $\bar{\nu} \in \nats$ such that $$\eta_\rho^{+,\nu}(g^\nu,G_0)\leq \delta \text{ for every } \nu\geq \bar{\nu}$$ and so $\varphi^{\nu}(g^\nu) < +\infty$ for sufficiently large $\nu$. Moreover,
	\[
	\begin{aligned}
		\varphi\left(f^{\star}\right) \leq \liminf_\nu \varphi^{\nu_{k}}\left(f^{k}\right) &\leq \liminf_\nu \left(\inf _{f \in \mathscr{F}^{\nu_{k}}} \varphi^{\nu_{k}}(f)+\varepsilon^{n_{k}}\right) \\
		&\leq \limsup_\nu \varphi^{\nu_{k}}\left(g^{\nu_{k}}\right)\\
		&\leq \varphi\left(g\right)=\inf_{f \in \mathscr{F}} \varphi(f)
	\end{aligned}
	\]
	which proves that $f^\star \in \operatorname{argmin}_{F \in \mathscr{F}}\varphi(F)$.
\end{proof}

\section{Numerical Examples}\label{sec:Numerical}

We next turn to the numerical procedure that leverages the previous developments. We set our examples in two-dimensional space ($m=2$) to avoid implementation challenges associated with higher dimensions, \red{ although the results in Section~\ref{sec:thms} remain true as stated}.  For given two-dimensional cdfs $F_0,\,G_0\in\cdfcns(S)$, consider the following pseudo-code:  For a given tolerance $\varepsilon>0$
\begin{enumerate}
	\item Construct an infinite refinement $\{\mathcal{R}^\nu\}_{\nu\in\nats}$ of $S$, where each element $\mathcal{R}^\nu$ is a box partition.
	\item Divide each rectangular element of the box partition $\mathcal{R}^\nu$ into two triangles and use epi-splines of degree~1 to approximate the elements on the space of usc functions.
	\item Solve the sequence of approximated optimization given by Problem~\ref{p_eta_rho_esplns_nu} for each $\mathcal{R}^\nu$.  
	
	For each $\nu$-instance, we solve these problems defined by the function $\eta_\rho^{+,\nu}$ by using binary search on an auxiliary variable representing feasibility of the stochastic ambiguity set.  Given our chosen epi-spline space, this procedure corresponds to solving a sequence of linear programming problems.  For more details, consult \cite{Fernanda}.
\end{enumerate}
We implemented our algorithm in Python, using Pyomo \cite{bynumpyomo} for modeling the optimization problems, and we use the commercial solver Cplex \cite{cplex2009v12}, with its academic license, for solving these problems.

In this section, we evaluate the algorithm's performance through two numerical examples, exploring various functions $F_0$, $G_0$, and $\delta$ values in Problem~\ref{problem2}. Here, $\delta$ represents the size of the ambiguity set and reflects the level of trust or proximity to the cdf $G_0$. 
The first example (Section~\ref{ssec:uuv}) involves estimating the position of an unmanned underwater vehicle using two sources of independent information (two distinct datasets). The second example (Section~\ref{ssec:2unif}) tackles the estimation problem with two uniform distributions, enabling us to assess different solutions for various ambiguity parameter values.

Moreover, we present an example including constraints modeling the shape of the cdf.  First, we introduce a \emph{bounded growth} constraint, which corresponds to bounding the local Lipschitz modulus for each element of the partition.  This constraint aims to obtain a smooth-looking cdf.
Next, we include a experiment where the distribution constraint \eqref{eq:distcond} is imposed only over the smallest rectangles of the partition elements.

For all experiments, a mesh of 100 points per axis and a specified tolerance of $\varepsilon = 10^{-8}$ are utilized, ensuring accurate and reliable results.  The datasets generated during the current study are available from the corresponding author on request.

\subsection{The UUV position estimation problem}\label{ssec:uuv}
Consider an unmanned underwater vehicle (uuv) that is returning to a docking station after a long mission. Although the uuv knows the location of the docking station on the map, it has only a vague idea about its own location because it has been underwater for a long time and has only used an inertial navigation system during that time. The docking station sends out pings that can be picked up by the uuv when close enough. The uuv can use these pings to improve the estimate of its own location. The uuv also has an accurate model (in the short term) of where it will be, given that it knows its initial condition.

Let $F_t$ be the cdf of the uuv location over the $(x_1,x_2)$ coordinates based on the inertial navigation system.  Ping data $y_t = (y_{t}^1,y_t^2)$ is informed by the docking station at time $t$. The ping data $y_t$ has noise, so it is not the true location of the uuv at time $t$.  We use a model (Dubin) to construct an estimate of the cdf of the uuv location based on the ping data, called $G_t$.  This serves as another source of information about the uuv location.

The problem is to find a cdf  that models the position of the uuv at time $t$.  Given the two sources of information, our estimation problem is posed to minimize the distance to $F_t$, while keeping the distance to $G_t$ bounded.
For each $k=1, \ldots, t$, we have $N$ data points about the location of the uuv. The true trajectory is considered as $\bar{x} = (\bar{x}^1,\ldots,\bar{x}^t)$, and with the ping data we compute $z^k$, $k=1, \ldots, t$ to construct $G_{t}$ from it.
\noindent The expected values  of position using the cdfs $F_t$ and of $G_t$ are $(8.8261,1.0724)$ and $(9.9965, 1.7468)$ respectively.  Table \ref{tab:uuv_results} illustrates the results obtained for different values of $\delta$.

\begin{table}[ht!]
	\centering
	\begin{tabular}{ccccc}
		\noalign{\smallskip}\hline\noalign{\smallskip}
		$\delta$&  $s$ &  $\eta$    &  Expected position &  Execution time [sec] \\
		\noalign{\smallskip}\hline\noalign{\smallskip}
		$0.90$  & 0.0000 & 0.0787 &  (9.034, 1.8908)  &  969  \\
		$0.10$ &    0.0000  &    0.3501      &   (9.734, 1.2908)  &   1008    \\
		$0.01$  & 0.0824     & 1.0    & (10.3023, 1.8908) &  1008 \\
		\noalign{\smallskip}\hline
	\end{tabular}
	\caption{Solution estimates for the uuv position estimation problem for several values of the ambiguity set radius ($\delta$).}
	\label{tab:uuv_results}
\end{table}

We depict the data sets and various expected locations for cdfs obtained for different values of $\delta$ in Figure~\ref{fig:plot_uuv}.  In this picture, blue markers corresponds to data points coming from the inertial navigation system and determine $F_t$. Grey markers are points associated to the modeling data set used to construct $G_t$.  Meanwhile, the black markers, labeled as $X_t$ and $Z_t$, represent the expected locations corresponding to cdfs $F_t$ and $G_t$ respectively.  The resulting expected positions, estimated for $\delta=0.9,0.1,0.01$, are visualized through the use of red markers.  This picture shows the influence of the ambiguity set radious on the estimation of the uuv position: as $\delta$ diminishes, the estimated position tends to converge towards the position determined by $G_t$.

\begin{figure}[h!]
	\centering
	\includegraphics[width = 0.75\textwidth]{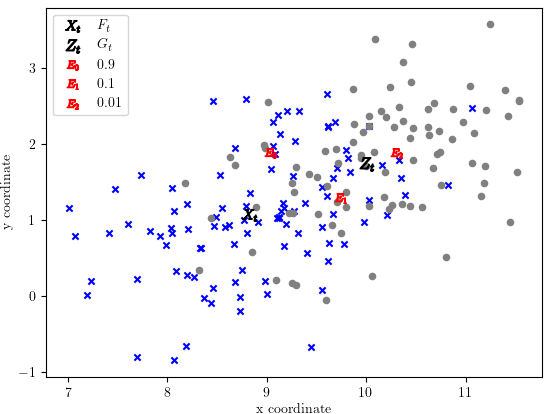}
	\caption{Expected positions (depicted as red markers $E_0,E_1,E_2$) resulting from various cdfs estimated by the algorithm for different uncertainty set radii, $\delta = 0.9, 0.1, 0.01$ respectively. The position denoted as $X_t$ and $Z_t$ correspond to the expected locations linked to $F_t$ and $G_t$ respectively.}
	\label{fig:plot_uuv}
\end{figure}

The parameter $\eta$ reported  in Table~\ref{tab:uuv_results} provides an estimate for $\hatdl_\rho(F_t,F^*)$, where $F^*$ is a solution yielded by the algorithm. We use the parameter $s$ as a tolerance to assess feasibility with respect to the stochastic ambiguity set, which is constructed as a ball centered at the cdf $G_t$.

The results in Table~\ref{tab:uuv_results} show that the estimated expected position approaches to the expected position under the cdf $G_t$ as the ambiguity set radius ($\delta$) decreases.  Additionally, the distance to $F_t$ increases as the ambiguity set shrinks, reaching the upper bound of $1$ at $\delta=0.01$.  This is because the algorithm is forced to choose a cdf that is closer to $G_t$ as $\delta$ decreases.

\subsection{Two uniform distributions}\label{ssec:2unif}
Let $F_0$ be a uniform cdf over $[0,1]\times[0,1]$ and let $G_0$ be a uniform cdf over $[2,3]\times[2,3]$.  Note that the supports of these cdfs are disjoint, and thus $\hatdl_\rho(F_0,G_0)=1$ for $\rho\geq 1$.  We solve the estimation problem given by Problem~\ref{p_eta_rho_esplns_nu}, of finding a cdf that minimizes the distance to $F_0$, while remaining in the ambiguity set centered at $G_0$ and with radius $\delta$, for several values of $\delta$.  We also impose the distribution constraint \eqref{eq:distcond} over every rectangle of the mesh.  Our estimation scheme relies on a binary search over the optimal objective function, reported by the variable $\eta$, and the feasibility auxiliary variable represented by $s$.  Note that, for small values of $\delta$, the ambiguity constraint leads to a solution that is feasible and closer to $G_0$, but which its distance to $F_0$ reaches the upper bound of 1.

Table~\ref{tab:setting2} outline the outcomes for a collection of values for the parameter $\delta$.  The corresponding outcomes for $\delta\in\{ 1.0,0.7,0.1\}$ underscore that no ambiguity constraint violations occur within these values, and thus the estimated cdfs lie within a distance of at least $\delta$ to $G_0$. For the value of $\delta=10^{-4}$, the feasibility parameter $s$ has a positive value, meaning that the estimated cdf can be assured to lie just at a distance of $\delta + s$.  Nonetheless, this value has been shown to decreased as the mesh gets finer (for meshes with more than 100 points in each axis). Moreover, for the same $\delta$ as before, the approximated distance, $\eta$, coincides with the theoretical value indicated in Figure~\ref{fig:s2}. Table~\ref{tab:setting2} also reports the quantity \emph{Error~\%}, that quantifies the relative percentage error of all possible rectangles that can be obtained from vertices in the mesh ($\mathcal{O}(N^4)$, $N$ being the number of points in each axis) that do not satisfy the distribution constraint.  In this numerical setting, it shows that errors are under 1\%. Note that the distribution constraint is only imposed on the rectangles of the mesh ($\mathcal{O}(N^2)$).

\begin{table}
	\centering
	\begin{tabular}{cccccccc}
		\hline\noalign{\smallskip}
		$\delta$ &  $s$ &   $\eta$ & Estimated expected value&  Error (\%) & Execution time [sec]   \\
		\noalign{\smallskip}\hline\noalign{\smallskip}
		1.00 & 0.0 & 0.019 & (0.4597, 0.5021) & 0.9811 & 3327     \\
		0.70 & 0.0 & 0.30 & (1.1400, 1.1382) &  0.3866 & 558       \\
		0.10 & 0.0 & 0.89 & (2.6385, 2.6441) &  0.0764 & 484        \\
		$10^{-4}$ &  0.058 & 1.0 & (2.6385, 2.6441)  &  0.0100 & 454   \\ 
		\noalign{\smallskip}\hline
	\end{tabular}
	\caption{Results for the cdf estimate with two disjoint uniform distribution, for values of $\delta$ in $\{1.00, 0.70, 0.10, 10^{-4}\}$.}\label{tab:setting2}
\end{table}

\begin{figure*}
	{\centering
		\includegraphics[width=0.4\textwidth]{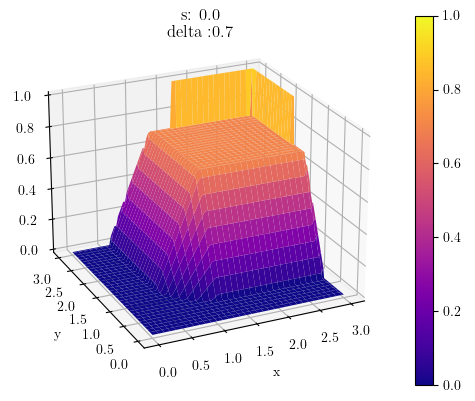}}
	\centering
	{\centering  
		\includegraphics[width=0.4\textwidth]{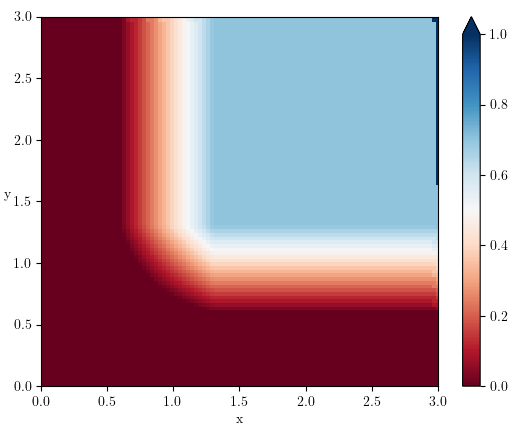}}
	{\centering
		\includegraphics[width=0.4\textwidth]{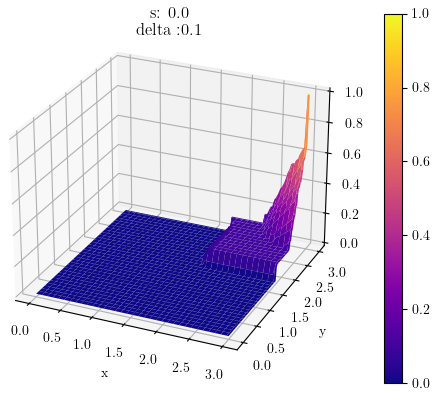}}
	\centering
	{\centering  
		\includegraphics[width=0.4\textwidth]{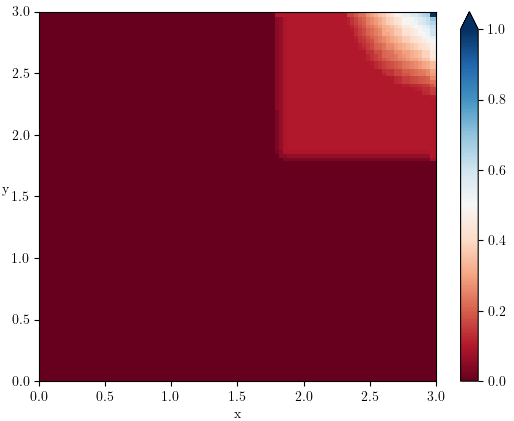}}
	\caption{The graph of the estimated cdfs for $\delta=0.7$ and $\delta=0.1$ are presented in the left column, along with their respective heatmaps showcased in the right column. Further insights and detailed metrics are available in Table~\ref{tab:setting2}.}\label{fig:s2}   
\end{figure*}

Our last case corresponds to two disjoint uniform cdfs, akin to the previous setting.  However, in this example we incorporate a bounded growth constraint to the solution.  The motivation behind this constraint lies in its ability to enhance the smoothness of the resulting cdf.  This constraint is stated over every triangular element of the partition, and bounds the growth of the distribution by a given parameter $L$.  This corresponds to a locally Lipschitz-type constraint over each element of the partition, keeping the Lipschitz modulus below the parameter $L$.  Note that this constraint is a local version of the global property. 

We execute this example for various combinations of  $\delta$ and $L$, specifically for $(\delta=0.70, L=1.0)$ and $(\delta=0.10, L=0.85)$ respectively.  The outcomes of these instances are laid out in Table~\ref{tab:setting2-bg}, while Figure~\ref{fig:s2_L1} depicts the resultant cdfs.  For the first instance, it easy to verify that the estimated cdf adopts a notably smoother character when compared to its unconstrained counterpart, as evident in the initial row of Figure~\ref{fig:s2}.  This observation resonates with the intention behind introducing the bounded growth constraint, as it successfully yields a smoother cdf outcome.

\begin{table}
	\centering
	\begin{tabular}{ccccccc}
		\hline\noalign{\smallskip}
		$\delta$ & $L$ &  $s$ &   $\eta$ &  Estimated expected value& Execution time [sec] \\
		\noalign{\smallskip}\hline\noalign{\smallskip}
		0.70 & 1.0 & 0.0 & 0.2999 & (0.80, 0.76) & 708 \\
		0.10 & 0.85 & 0.0024 & 1.0 & (2.40, 2.40) & 425\\
		\noalign{\smallskip}\hline
	\end{tabular}
	\caption{Results for $(\delta=0.70,L=1.0)$ and $(\delta=0.10,L=0.85)$, when adding the growth constraint to the two disjoint uniform cdfs example. The table shows the values of the auxiliary feasibility variable $s$, auxiliary objective function estimation $\eta$, the estimated expected value, and the execution time for each case.}\label{tab:setting2-bg}
\end{table}

\begin{figure*}
	{\centering
		\includegraphics[width=0.40\textwidth]{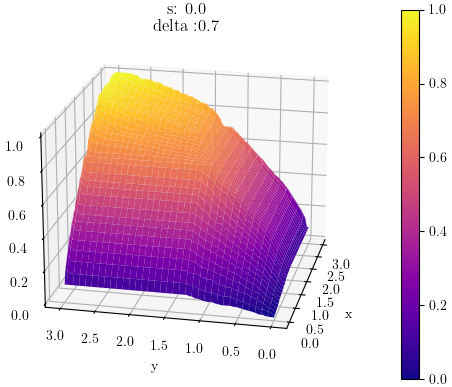}}
	\hfill
	\centering
	{\centering  
		\includegraphics[width=0.40\textwidth]{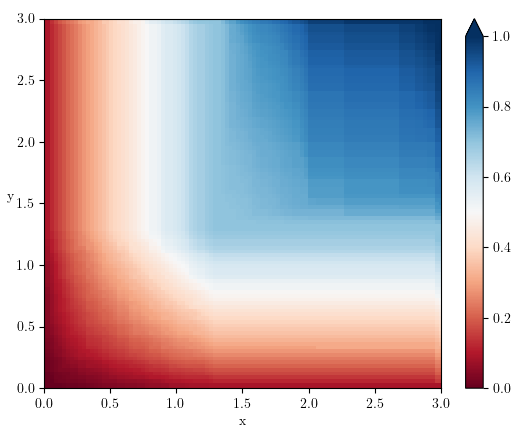}}
	\caption{Graph and heat map of the estimated cdf, for values of $(\delta = 0.7, L = 1)$.  The cdf incorporates the growth constraint, and is a smoother version of the unconstrained counterpart in the first row of Figure~\ref{fig:s2}.}\label{fig:s2_L1}
\end{figure*}

\medskip

\state Funding.
Julio Deride and Fernanda Urrea were partially supported by ANID under grant FONDECYT 11190549.  Johannes Royset was partially funded by Office of Naval Research under grants N0001421WX00142 and N00014-24-1-2318.

\medskip

\state Data availability.
All data in this manuscript was randomly generated as described in the document. The data is available from the corresponding author upon request.

\medskip

\state Competing interests.  
Johannes Royset is an editor in this journal.

\bibliography{references.fixed}
\bibliographystyle{plain} 

\end{document}